\documentclass[12pt,leqno]{amsart}
\usepackage{amsmath,amssymb,txfonts}
\usepackage{amssymb}
\usepackage{amsxtra}
\usepackage{amsmath}
\usepackage{txfonts}

 \usepackage{color}

\usepackage[cp1252]{inputenc}

\usepackage{mathrsfs}
\usepackage{graphicx}

 \usepackage[active]{srcltx}

\allowdisplaybreaks

\def\rr{{\mathbb R}}
\def\rn{{{\rr}^n}}

\def\zz{{\mathbb Z}}
\def\nn{{\mathbb N}}

\def\fz{\infty}

\def\dist{{\mathop\mathrm{\,dist\,}}}
\def\loc{{\mathop\mathrm{\,loc\,}}}

\def\lz{\lambda}

\def\ez{\epsilon}

\def\bz{\beta}

\def\gz{{\gamma}}

\def\boz{{\Omega}}

\def\vz{\varphi}

\def\ls{\lesssim}

\def\diam{{\mathop\mathrm{\,diam\,}}}

\def\r{\right}
\def\lf{\left}

\newtheorem{thm}{Theorem}[section]
\newtheorem{lem}{Lemma}[section]

\newtheorem{rem}{Remark}[section]

\numberwithin{equation}{section}

\begin{document}
\arraycolsep=1pt

\title[Fractional Orlicz-Sobolev extension/imbedding on Ahlfors $n$-regular domains ]{
Fractional Orlicz-Sobolev extension/imbedding on Ahlfors $n$-regular domains }

\author{Tian Liang}
\address{ Department of Mathematics, Beihang University, Beijing 100191, P.R. China}
                    \email{ liangtian@buaa.edu.cn}

\thanks{     }

\date{\today }
\maketitle

\begin{center}
\begin{minipage}{13.5cm}\small{\noindent{\bf Abstract}\quad
In this paper we build up  a criteria for fractional Orlicz-Sobolev extension and imbedding domains on Ahlfors $n$-regular domains.

}
\end{minipage}
\end{center}

\section{Introduction\label{s1}}

 The study of extension/imbeddings of   function spaces (including Sobolev,   BMO,
 Besov and Triebel-Lizorkin spaces)  and their applications  in   harmonic analysis,  potential theory and
 partial differential equations  have attracted a lot attentions; see for example  \cite{jw78,j80,j81,jw84,s70,ds93,k98,r99,t02,t08,hkt08,s06,s07,s10,z14}.

  In this paper, we are interested in the fractional Orlicz-Sobolev extension/imbedding.
Let $\phi $  be a  Young function, that is, $\phi\in C([0,\fz))$ is  convex,
 $\phi(0)=0$ and $\phi(t)>0$ for $t>0$.
 For any $\bz>0$ and  domain  $\Omega\subset\rn $, define the fractional Orlicz-Sobolev spaces  $\dot{ W }^{\beta,\phi}(\Omega)$  as  the space  of
  all $u\in L^1_\loc(\Omega) $   whose (semi-)norm
\begin{eqnarray*}
\|u\|_{\dot{ W }^{\beta,\phi}(\Omega)} := \inf \left\{\lambda > 0 : \int_{\Omega} \int_{\Omega}  \phi \left(\frac{|u(x)-u(y)|}{\lambda}\right) \frac{dxdy}{|x-y|^{n+\beta}}\leq 1 \right\}
\end{eqnarray*}
is finite.   Modulo constant functions, $\dot{ W }^{\beta,\phi}(\Omega)$ is a Banach space.
 If  $ \phi(t)=t^p$ with $p\ge1 $,   then $\dot{ W }^{\beta,\phi}(\Omega)=\dot{W}^{\beta/p,p}(\Omega) $.
Here   $\dot{W}^{s,p}(\Omega) $ with $s>0$ and $p\ge1$ is the fractional Sobolev space, that is,  the collection of
 all   $u\in L^1_\loc(\Omega)$ with
  \begin{eqnarray*}
\|u\|_{\dot{ W }^{s,p}(\Omega)} :=   \lf(\int_{\Omega} \int_{\Omega}  \frac{|u(x)-u(y)|^p}{|x-y|^{n+sp}} \,dxdy\r)^{1/p}<\fz.
\end{eqnarray*}

To guarantee the nontriviality of $\dot{ W }^{\beta,\phi}(\Omega)$,  we always assume
\begin{equation}\label{delta0}C_\bz:=\sup_{t>0}
\frac{t^\beta}{\phi(t)}\int_0^t\frac{\phi( s  )}{ s^\beta}\frac{ds}{s }<\infty.
\end{equation}
Indeed, \eqref{delta0}
implies that  $C^\fz_c(\Omega)\subset
\dot{ W }^{\beta,\phi}(\Omega)$; see Lemma \ref{xx}. Moreover,  \eqref{delta0} is optimal to  guarantee the nontriviality of $\dot{ W }^{\beta,\phi}(\Omega)$ in the sense that
 if   $\phi(t)=t^p$ with $p \geq 1$,   then $\dot{ W }^{\beta,\phi}(\Omega)=\dot{W}^{\beta/p,p}(\Omega)$ is nontrivial
 if and only if $p>\beta$ (see \cite{gkz13}), and if and only if $C_\bz<\fz$.
Besides of $\phi(t)=t^p$ with $p \geq 1$ and $p > \beta$, we refer to Remark \ref{re.1.1} for more Young functions satisfying \eqref{delta0}, in particular, including
 $\phi(t)=t^p[\ln(1+t)]^\alpha$ with $p> \beta$, $p \geq 1$ and $\alpha\ge 1$.
We  remark that under  \eqref{delta0}, $\dot{ W }^{\beta,\phi}$ has fractional smoothness strictly less than 1.

The main purpose of this paper is to build up the following criteria  for
fractional Orlicz-Soblev $\dot{ W }^{\beta,\phi}$-extension and -imbedding domains 
 when $\bz\in (0,n)\cup(n,\fz)$, which generalize the corresponding results for
fractional Sobolev spaces (see \cite{jw78,jw84,s06,s07,z14}).
We also note that the case  $\bz=n$ has already been considered  in \cite{lz18}.

\begin{thm}\label{t1.2}
Let $\bz\in(0,n)\cup(n,\fz)$ and
$\phi$ be a Young function satisfying   \eqref{delta0}. Suppose that
  $\phi$ is doubling, that is, there exists a constant $K > 1$ such that $\phi(2t) \leq K \phi(t)$ for all $t>0$.
For any domain $\Omega\subset\rn$,  the following are equivalent:
\begin{enumerate}
 \item[(i)] $\Omega$ is Ahlfors $n$-regular, that is,    there exists a constant $\theta>0$ such that
$$
|B(x,r)\cap\Omega|\ge \theta r^n\quad\forall x\in\Omega, 0<r<2\diam\Omega.$$

\item[ (ii)] $\Omega$ is a ${\dot W }^{\beta,\phi}$-extension domain, that is,    any function
$u \in \dot{ W }^{\beta,\phi}(\Omega)$ can be extended to a function $\tilde{u}\in {\dot W }^{\beta,\phi}(\mathbb R^n)$  in a continuous and linear way.

 \item[(iii)] $\Omega$ is a ${\dot W }^{\beta,\phi}$-imbedding domain, that is,
 \begin{enumerate}
\item[(a)] when $0<\beta < n$, there exists a constant $C = C(\beta, n,\phi )>0$ such that 
\begin{align}\label{im1}
\inf_{c\in\rr}\|u-c\|_{L^{\phi^{n/(n-\beta)}}(\Omega)}\leq C \|u\|_{\dot{ W }^{\beta,\phi}(\Omega)} \quad\forall u \in \dot{W}^{\bz,\phi}(\Omega);
\end{align}

\item[(b)]  when $\beta > n$, there exists a constant $C= C(\beta, n,\phi ) >0$ such that for each $u \in \dot{W}^{\bz,\phi}(\Omega)$,
 we can find a function $\tilde{ u}$ with $  \tilde{u}=u$ almost everywhere and
\begin{align}\label{im2}
|\tilde{u}(x)- \tilde{u}(y)|  \leq   C  \phi^{-1}\left( |x-y|^{\beta-n}\right)\|u\|_{\dot{ W }^{\beta,\phi}(\Omega)}\quad\forall x,y\in\Omega.
\end{align}
\end{enumerate}
\end{enumerate}
\end{thm}

Above  we denote by $L^\phi(\Omega)$  the Orlicz space, that is,
  the collection of all   $u\in L^1_\loc(\Omega)$ whose norm
 \begin{align*}
\|u\|_{ { L}^{\phi}(\Omega)} := \inf \left\{\lambda > 0 : \int_{\Omega} \phi \left(\frac{|u(x)|}{\lambda}\right) dx\leq 1 \right\}<\fz.
 \end{align*}

This paper is organized as follows.  In Section 2, we recall some properties of Young functions and show the nontriviality of $W^{\bz,\phi}$ under \eqref{delta0}.  The proofs of (i)$\Rightarrow$(ii), (ii)$\Rightarrow$(iii)
and  (iii)$\Rightarrow$(i) of Theorem \ref{t1.2} are given in Section 4, 3, 5 separately.


 Note that (i)$\Rightarrow$(ii) follows from the following Theorem \ref{t1.1}, where the doubling condition on $\phi$ is not needed.
Theorem \ref{t1.1} will be proved by an argument similar to the case $\bz=n$ as given in \cite{lz18},
but, due to some technical differences caused by $\bz\ne n$, we give the details in Section \ref{s4} for reader's convenience.
\begin{thm}\label{t1.1}
 Let  $\bz\in(0,n)\cup(n,\fz)$ and   $\phi$ be a Young function satisfying  \eqref{delta0}.
  If   $\Omega \subset \rn $ is  an Ahlfors $n$-regular domain,
then    $\Omega$ is a ${\dot W }^{\bz,\phi}$-extension domain.
\end{thm}

 The proof of (ii)$\Rightarrow$(iii) is given in Section \ref{s3}.
 When $\bz>n$, we use a $(1,\phi)_\bz$-Poincar\'e inequality proved in Lemma \ref{l3.1}; see Section \ref{s3.1}.
 When $\bz\in(0,n)$, the proof relies on the following  $(\phi^{n/(n-\bz)}, \phi)_\bz$-Poincar\'e   inequality.
 \begin{thm}\label{l3.9}
 Let $\bz\in(0,n)$ and
$\phi$ be a Young function satisfying   \eqref{delta0}. Suppose that
  $\phi$ is doubling. Then there exists a constant $C=C(\bz,n,\phi)>0 $ such that
   $$
\inf_{c\in\rr}\|u-c\|_{L^{\phi^{n/(n-\beta)}}(B)} \leq C \|u\|_{ \dot W ^{\beta,\phi}(B)}\quad\forall u \in \dot W ^{\beta,\phi}(B)
$$
whenever $B$ is a ball of $\rn$ or $B=\rn$.
\end{thm}

The $(\phi^{n/(n-\bz)}, \phi)_\bz$-Poincar\'e  inequality is a self-improvement of the $(1, \phi)_\bz$-Poincar\'e  inequality(see Lemma \ref{l3.1}).
One may wish to prove Theorem \ref{l3.9} via some known self-improvement approach from harmonic analysis.
%
But since here  Orlicz norm and fractional derivative are involved, the proof would be very complicated.
Instead, as motivated by \cite{npv12}, by building up  a local version of the known  geometric inequality
$$
\int_{\rn\setminus E } \frac{dy}{|x-y|^{n+\beta}} \geq C(n,\beta)|E|^{-\beta/n} \quad \mbox{whenever $|E|<\fz$  },
$$
  using some ideas from \cite{npv12} and also the median value,  we give a  direct proof of Theorem \ref{l3.9}.
See Section 3.2 for details.
%

To prove (iii)$\Rightarrow$(i) of Theorem \ref{t1.2}, the imbedding  assumption allows us to
calculate the  $\|u\|_{{\dot {W}}^{\bz,\phi}(  \Omega)}$-norm of some test functions.
 Using these and the doubling property of $\phi$, and following the idea  from \cite{hkt08} (see also \cite{hkt08b,z14}), we conclude that $\Omega$ is Ahlfors $n$-regular.

Notation used in the following is standard.
Denote by  $C$ some constant which may vary from line to line  but is independent of the main parameters.
The constant $C(X,Y,....)$  depends only on the parameters $X$,  $Y$, $\cdots$;
while the constant  with subscripts would not change in different occurrences, like $C_1$.
The symbol $A \lesssim B$ means that $A \leqslant CB$. For any locally integrable function $u$ and measurable set $X$ with $|X|>0$, we denote by $u_X$ the average of $u$ on $X$, namely $u_X = \fint _{X} u \equiv \frac{1}{|X|} \int _{ X} u \, dx$. We use $d(x, E)$ to describe the Euclidean distance from $x$ to a set $E$.

\section{Preliminaries\label{s2}}

The following properties of Young functions are well-known, but for the convenience of the reader, we give the proof.
\begin{lem}\label{l2.4}
Let $\phi$ be a Young function.
 \begin{enumerate}

  \item[(i)]   Then $\phi$ is continuous, strictly increasing and $\lim_{t\to\fz}\phi(t)=\fz$.

 \item[(ii)]   The    inverse $\phi^{-1}$ of $\phi$ is continuous, strictly increasing,   $\phi^{-1}(0)=0 $ and $\lim_{t\to\fz}\phi^{-1}(t)=\fz$.
Moreover, $\phi^{-1}$ is concave; in particular, $\phi^{-1}(2x)\le 2\phi^{-1}(x)$ for all $x>0$.

 \item[(iii)]  If $\phi$ is doubling with some constants $K>1$, then $\phi^{-1}(tx)\le t^{1/(K-1)}\phi^{-1}(x)$ for all $x >0$ and $t \in [0, 1]$. 
 \end{enumerate}
\end{lem}

\begin{proof} (i)  Since $\phi$ is  convex, we have $\phi(t)\ge t\phi(1)$ for all $t\ge1$ and hence $\lim_{t\to\fz}\phi(t)=\fz$.
  Also note that   $\phi'(t)\ge0$ for almost all $t\ge0$ and $\phi'$ is increasing. If $\phi$ is not strictly increasing, we must have $\phi(t)=\phi(t+s)$ for some $t\ge 0$ and $s>0$.  Thus
 $\phi'=0$ almost everywhere in $[t,t+s]$ and hence in $[0,t+s]$.
 This implies that $\phi=0$ in $[0,t+s]$, which contradicts with $\phi>0$ in $(0,\fz)$.

(ii) From (i) we see easily that the inverse $\phi^{-1}$  is continuous, strictly increasing,   $\phi^{-1}(0)=0 $ and $\lim_{t\to\fz}\phi^{-1}(t)=\fz$.
Moreover,
for any $x,y \ge0$,
from the convexity of $\phi$ it follows that for any $\lz \in [0,1]$,
\begin{align*}
\phi^{-1}(\lz x +(1-\lz)y)& =\phi^{-1}(\lz\phi(\phi^{-1}(x))+(1-\lz)\phi(\phi^{-1}(y)))\\
&\geq \phi^{-1} (\phi (\lz \phi^{-1}(x)+(1-\lz)\phi^{-1}(y)))=  \lz \phi^{-1}(x)+(1-\lz)\phi^{-1}(y).
\end{align*}
Thus $\phi^{-1}$ is concave.
Furthermore, thanks to $\phi^{-1}(0)=0 $ and the concavity, we get $\phi^{-1}(2x)\le 2\phi^{-1}(x)$ for all $x>0$.

(iii) Since   $\phi'$ is increasing, we have
$$
\phi(2t)-\phi (t)= \int_t^{2t}\phi'(s)\,ds\ge\phi'(t)t\quad\forall t>0.$$
By the doubling property of $\phi$, we have
$\phi(2t)-\phi(t)\le (K-1)\phi(t)$
and hence,
$$(K-1)\phi(t)\ge \phi'(t)t,\ \mbox{that is,\ }\ (\ln\phi)'(t)=\frac{\phi'(t)}{\phi(t)}\le\frac{K-1}t\quad\mbox{for almost all $t>0$.}$$
Thus for $t\in(0,1]$ we have
 \begin{equation*}
\ln\left(\frac{\phi(x)}{\phi(tx) }\right)= \int_{tx}^{x} (\ln\phi)'(s) \,ds \leq \int_{tx}^{x} \frac{K-1}{s}\,ds = \ln \left(\frac{1}{t^{K-1}}\right),
\end{equation*}
which gives that
$$t^{K-1 }\phi(x)\le \phi(tx).$$

For any $t>0$, $\phi^{-1}(\phi(t))=t$ implies $(\phi^{-1})'(\phi(t))\phi'(t)=1$ for almost all $t>0$,
 that is,
  $$(\phi^{-1})'(\phi(t))=\frac1{\phi'(t)} \ge \frac1{\phi(t)}\frac t{K-1}.$$
We  then have
$$(\phi^{-1})'(s)\ge  \frac1{s}\frac {\phi^{-1}(s)}{K-1},
\quad \mbox{that is,} \quad  (\ln \phi^{-1})'(s)\ge \frac1{s(K-1)}
 \quad \mbox{for almost all $s>0$}.$$
Thus for $t\in(0,1]$ we have
 \begin{equation*}
\ln\left(\frac{\phi^{-1}(x)}{\phi^{-1}(tx) }\right)= \int_{tx}^{x} (\ln\phi^{-1})'(s) \,ds \geq \int_{tx}^{x} \frac1{s(K-1)}\,ds = \ln \left(\frac{1}{t^{1/(K-1)}}\right),
\end{equation*} which gives that
$$
 t^{1/(K-1) }\phi^{-1}(x) \ge  \phi^{-1}(tx) $$
 as desired.
\end{proof}

\begin{lem}\label{xx} Let $\beta>0$ and $\phi$ be a Young  function satisfying \eqref{delta0}.
For any domain $\Omega\subset\rn$,  $C_c^1(\Omega)\subset \dot{W}^{\bz,\phi}(\Omega)$.
\end{lem}
\begin{proof}
Given any $u \in C_c^1(\Omega)$, assume  $L=\| u\|_{L^\fz(\Omega)} +\|Du\|_{L^\fz(\Omega)}>0$ and choose a domain $W \subset \Omega$ such that $V={\rm\,supp\,}u\Subset W\Subset\Omega$. Then
\begin{align*}
H:=\int_\Omega\int_\Omega\phi\left(\frac{|u (z)-u (w)|}\lambda\right)\frac{dzdw}{|z-w|^{n+\beta}}
&\le  \int_W\int_{W}\phi\left(\frac{| z-w|}{\lambda/ L} \right)\frac{dzdw}{|z-w|^{n+\beta}} + 2
\int_{\Omega\setminus W}\int_{V}\phi\left(\frac  {L}{\lambda } \right)\frac{dzdw}{|z-w|^{n+\beta}}.
\end{align*}
By \eqref{delta0}, we have
\begin{align*}
 \int_W\int_W\phi\left(\frac{| z-w|}{\lambda/L} \right)\frac{dzdw}{|z-w|^{n+\beta}}&\le \int_V\int_{B(w,2|\diam W|)}\phi\left(\frac{| z-w|}{\lambda/L} \right)\frac{dz}{|z-w|^{n+\beta}}dw\\
&=n\omega_n\int_W\int_0^{2|\diam W|}\phi\left(\frac{t}{\lambda/L} \right)\frac{ dt}{t^{\beta+1}}\,dw\\
&=n\omega_n|W| (\frac{L} \lambda)^\beta \int_0^{2L|\diam W|/\lambda}\phi\left(s\right)\frac{ ds}{s^{\beta+1}} \\
&=n\omega_n C_{\beta} |W| 2^{-\beta}|\diam W|^{-\beta}\phi\left(\frac{2L|\diam W|}{\lambda } \right).
\end{align*}
Moreover,
\begin{align*}
2\int_{\Omega\setminus W}\int_{V}\phi\left(\frac  1{\lambda/L } \right)\frac{dzdw}{|z-w|^{n+\beta}}
&\le 2\phi\left(\frac  L{\lambda  } \right)\int_{V}\int_{\Omega\setminus B(z,\dist(V,W^\complement) )}\frac{dwdz}{|z-w|^{n+\beta}}
 \le 2\omega_n \phi\left(\frac  {L}{\lambda }\right)|V|\dist(V,W^\complement)^{-\beta}.
\end{align*}
Letting $\lambda$ large enough and using the convexity of $\phi$, we have $H\le1$.  That is, $u \in \dot{W}^{\bz,\phi}(\Omega)$ as desired.
\end{proof}

\begin{rem}\rm\label{re.1.1}
Let $\phi(t) = t^{p} \psi (t)$ be a Young function with $p \geq 1$. If $p > \beta$ and there exists a constant $C>0$ such that $\psi(s) \leq  C \psi(t)$ for all $s\leq t$, then $\phi$ satisfies  the  condition \eqref{delta0}, that is, $C_\beta < \infty$.
Indeed, for any $t> 0$,
\begin{align*}
\frac{t^\beta}{t^{p}\psi(t)}\int_0^t\frac{t^{p}\psi( s  )}{ s^\beta}\frac{ds}{s } \leq Ct^{\beta-p} \int_{0}^{t} s^{p-\beta-1} ds \leq \frac{C}{p-\beta}.
\end{align*}
Below are some  typical examples of Young functions $\phi$ satisfying \eqref{delta0}.

(i) $\phi(t)=t^p[\ln(1+t)]^\alpha$ with $p> \beta$, $p \geq 1$ and $\alpha\ge 1$.

(ii)$\phi(t)=\max\left\{t^p, t^{p+\delta}\right\}$ where $p \geq 1$, $p >\beta$ and $\delta > 0$.

(iii)$\phi(t)=t^pe^{ct^\alpha}$  with $p> \beta$, $p \geq 1$, $c>0$ and $\alpha> 0$.

(iv)$\phi(t)=e^{c t^\alpha }-\sum_{j=0}^{[n/\alpha] }(ct^\alpha)^j/j!$ where $p \geq 1$, $c>0$ with $\alpha>0$,
where $[n/\alpha]$ is  the integer less than or equal to $n/\alpha$.

Note that the Young functions given in (i) and (ii) further satisfy the doubling property, but the Young functions given in (iii) and  (iv) do not.
\end{rem}

\section{Proofs of Theorem \ref{l3.9} and  (ii)$\Rightarrow$(iii) of Theorem 1.1}\label{s3}

In Section \ref{s3.1} we prove (ii)$\Rightarrow$(iii) of Theorem \ref{t1.2} when $\bz\in(n,\fz)$.
  In Section \ref{s3.2} we prove Lemma \ref{l3.9} and then (ii)$\Rightarrow$(iii) of Theorem \ref{t1.2}  when $\bz\in(0,n)$.
    Below, we always denote by $\omega_n$ the $(n-1)$-dimensional Lebesgue measure of    the unit sphere $S^{n-1}$.
 \subsection{Case $\bz\in(n,\fz)$}\label{s3.1}

First, we have the following $(1,\phi)_\bz$-Poincar\'e inequality.
\begin{lem}\label{l3.1} Let $\beta>0$ and  $\phi$ be a Young  function satisfying \eqref{delta0}.
For any ball $B=B(z,r) \subset \rn$ and $u \in \dot{ W }^{\beta,\phi}(B)$, we have
$$
\fint_{B} |u(x)- u_{B}| \, dx \leq \phi^{-1} (2^{n+\beta}r^{\beta-n}\omega_{n}^{2}) \|u\|_{\dot W^{\bz,\phi}(B)}.
$$
\end{lem}

\begin{proof}
Let $u \in \dot W^{\bz,\phi}(B) $. For any ball $B \subset \Omega$ and $\lambda > \|u\|_{\dot W^{\bz,\phi}(\Omega)} $, applying  Jensen's inequality, we know
\begin{eqnarray*}
\phi \left(\frac{ \fint_{B} |u(x)- u_{B}|\,dx}{\lambda}\right)
&&\le  \fint_B \fint_B \phi \left(\frac{|u(x)-u(y)|}{\lambda}\right)\,dydx\\
&&\le  2^{\beta+n}\omega_n^2 r^{\beta-n}\int_B \int_B \phi \left(\frac{|u(x)-u(y)|}{\lambda}\right)\,\frac{dydx}{|x-y|^{n+\beta}}\\
&&\le 2^{\beta+n}\omega_n^2 r^{\beta-n},
\end{eqnarray*}
that is,
$$\fint_{B} |u(x)- u_{B}|\,dx\le \lambda \phi^{-1}(2^{\beta+n}\omega_n^2 r^{\beta-n}).$$
Letting $\lambda\to \|u\|_{\dot{ W }^{\beta,\phi}(B)}$, we obtain
$$\fint_{B} |u(x)- u_{B}|\,dx\le  \phi^{-1}(2^{\beta+n}\omega_n^2 r^{\beta-n})\|u\|_{\dot{ W }^{\beta,\phi}(B)}$$ as desired.
\end{proof}

Applying Lemma \ref{l3.1} and Lemma \ref{l2.4}, we obtain the following imbedding.
\begin{lem}\label{l3.2}  Let $\beta>n$ and  $\phi$ be a Young  function satisfying \eqref{delta0} and the doubling property with constant $K$.
There exists  a positive constant $ C(\beta, n )$ depending only on  $n$ and $\beta$  such that a
for all $u \in \dot{ W }^{\beta,\phi}(\rn)$, we can find a continuous function $\hat{u} \in \dot{ W }^{\beta,\phi}(\rn) $ such that $\hat{u}=u \  a.\,e. $ and
\begin{equation}\label{3.yy1}
|\hat{u}(x)-\hat{ u}(y)| \leq   C(\beta, n) \phi^{-1}\left(|x-y|^{\beta-n}\right)\|u\|_{\dot{ W }^{\beta,\phi}(\rn)} \quad \forall x,y \in \rn.
\end{equation}
\end{lem}

\begin{proof}
 Let $u \in \dot{ W }^{\beta,\phi}(\rn)$. We first show that
\begin{equation}\label{3.yy2}
|u(x)-u(y)| \leq  C(\beta, n) \phi^{-1}\left(|x-y|^{\beta-n}\right)\|u\|_{\dot{ W }^{\beta,\phi}(\rn)}
\end{equation}
for all   Lebesgue points $x,y$ of $u$.
Write
\begin{eqnarray*}
|u(x)- u(y)| \leq |u(x)- u_{B(x,\, 2|x-y|)}| + |u_{B(x,\, 2|x-y|)}- u(y)| .
\end{eqnarray*}
By Lemma \ref{l3.1},
\begin{eqnarray*}
|u(x)- u_{B(x,\, 2|x-y|)}|
&& \leq \sum \limits _{i=0}^{\infty} |u_{B(x,\, 2^{-i+1}|x-y|)}-u_{B(x,\, 2^{-i}|x-y|)}|\\
&& \leq 2^{n} \sum \limits _{i=0}^{\infty} \fint_{B(x,\, 2^{-i+1}|x-y|)}|u(z)-u_{B(x,\, 2^{-i+1}|x-y|)}|\,dz\\
&& \leq  2^{n} \sum \limits _{i=0}^{\infty}\phi^{-1}\left(2^{\beta+n}\omega_n^2 (2^{-i}|x-y|)^{\beta-n}\right)\|u\|_{\dot{ W }^{\beta,\phi}(B(x,\, 2^{-i+1}|x-y|))}\\
&&  \leq  2^{n} \sum \limits _{i=0}^{\infty}\phi^{-1}\left(2^{\beta+n}\omega_n^2 (2^{-i}|x-y|)^{\beta-n}\right)\|u\|_{\dot{ W }^{\beta,\phi}(\rn)}.
\end{eqnarray*}
Thanks to Lemma \ref{l2.4}(iii), we have
\begin{eqnarray*}
 \phi^{-1}\left(2^{\beta+n}\omega_n^2 (2^{-i}r)^{\beta-n}\right)
 \leq 2^{-i(\beta-n)/(K-1)} \phi^{-1}\left(2^{\beta+n}\omega_n^2 |x-y|^{\beta-n}\right)
\leq   2^{-i(\beta-n)/(K-1)}2^{\beta+n}\omega_n^2  \phi^{-1}\left(|x-y|^{\beta-n}\right) .
\end{eqnarray*}
Thus
\begin{eqnarray*}
|u(x)- u_{B(x,\, 2|x-y|)}|
\leq   \sum \limits _{i=0}^{\infty} 2^{-i(\beta-n)K} 2^{\beta+2n}\omega_n^2 \phi^{-1}\left( |x-y|^{\beta-n}\right)\|u\|_{\dot{ W }^{\beta,\phi}(B)}
\leq C(\beta, n)\phi^{-1}\left( |x-y|^{\beta-n}\right)\|u\|_{\dot{ W }^{\beta,\phi}(\rn)}.
\end{eqnarray*}
Similarly,
 $$|u(y)- u_{B(x,\, 2|x-y|)}|\leq  C(\beta, n)\phi^{-1}\left(|x-y|^{\beta-n}\right)\|u\|_{\dot{ W }^{\beta,\phi}(\rn)}.$$

Next, let $\hat{u}(x):= \lim \limits _{r \rightarrow 0} u_{B(x,r)}    $ for all $x\in\rn$.
Note that $\hat u$ is well-defined. Indeed, for any $0<r <s $, by \eqref{3.yy2},   we have
\begin{align*}
| u_{B(x,r)}- u_{B(x,s)}|
&\le  \fint_{B(x,s)}  \fint_{B(x,r)}|u(z)-  u(w)|\, dz\,dw \leq  C(\beta, n )\phi^{-1}\left( (r+s)^{\beta-n}\right)\|u\|_{\dot{ W }^{\beta,\phi}(\rn)},
\end{align*}
which  together with the continuity of $\phi^{-1}$ and $\phi^{-1}(0)=0$ implies the existence of $\hat u(x)$.
Due to the Lebesgue differentiation theorem, we know $\hat{u} =u $ almost everywhere, and hence, $\hat{u} \in \dot{ W }^{\beta,\phi}(\rn)$.
Moreover,  $\hat{u}$ is continuous; indeed, by \eqref{3.yy2}
\begin{align*}
|\hat{u}(x)-\hat{ u}(y)|& =\lim _{r \to0}\left|u_{B(x,r )}   - u_{B(y,r)} \right|\\
&
 \leq  \lim _{r \to0}\fint_{B(x,2r+|x-y|)}  \fint_{B(x,r)}|u(z)-  u(w)|\, dz\,dw
 \leq   C(\beta, n) \phi^{-1}\left(  |x-y|^{\beta-n}\right)\|u\|_{\dot{ W }^{\beta,\phi}(\rn)}.
 \end{align*}
 This completes the proof of Lemma \ref{l3.2}.
\end{proof}

\begin{proof}[Proof of (ii) $\Rightarrow$ (iii) Theorem \ref{t1.2}: case $\bz\in(n,\fz)$]
Since $\Omega$ is  a ${\dot W }^{\beta,\phi}$-extension domain,
for every  $u\in \dot{ W }^{\beta,\phi}(\Omega)$, we can find a $ \tilde{u}\in \dot{ W }^{\beta,\phi}(\rn)$ such that $\tilde{ u}=u$ in $\Omega$ and $\| \tilde u {} \|_{\dot{ W }^{\beta,\phi}(\rn)  }\le  C\|u\|_{\dot{ W }^{\beta,\phi}(\Omega)}$.

Using Lemma \ref{l3.2}, there exists continuous function $\hat{\tilde{u}}\in  \dot{ W }^{\beta,\phi}(\rn)$ such that
$\hat{\tilde{u}} = \tilde{u}$ almost everywhere in $\rn$ and for all $x,y\in\rn$,
\begin{align*}
|\hat{\tilde{u}}(x)- \hat{\tilde{u}}(y)|
& \leq   C(\beta, n) \phi^{-1}\left( |x-y|^{\beta-n}\right)\|\tilde{u}\|_{\dot{ W }^{\beta,\phi}(\rn)}
\lesssim  \phi^{-1}\left( |x-y|^{\beta-n}\right)\|u\|_{\dot{ W }^{\beta,\phi}(\Omega)}.
\end{align*}
Therefore we have $\hat{\tilde{u}}=u$ almost everywhere in $\Omega$ and \eqref{im2} holds.
\end{proof}

\subsection{Case $\bz\in(0,n)$}\label{s3.2}
In order to prove Theorem   \ref{l3.9}, we need  the following geometry inequality. 

\begin{lem}\label{l3.5}
 Let $\beta \in (0,n)$. Then there exists a constant $C(n,\beta)$ depending only on $n$ and $ \beta$ such that  for any ball $B \subset \rn$ and $x\in B$, we have
$$
\int_{B\setminus E } \frac{dy}{|x-y|^{n+\beta}} \geq C(n,\beta)|E|^{-\beta/n} \mbox{whenever $E \subset B $ and $0<|E|<\frac12 |B|$}.
$$
\end{lem}

\begin{proof}
Let $\rho \in (0, 2\diam B)$ such that $|E|=|B \cap B(x,\rho)|\leq \omega_n \rho^{n}$. Then
\begin{align*}
|(B \setminus E)\cap B(x,\rho)|=|B \cap B(x,\rho)|-|E \cap B(x,\rho)|
& = |E|- |E \cap  B(x,\rho)|= |E \cap  B^{\complement}(x,\rho)|.
\end{align*}
Hence
\begin{align*}
\int_{B\setminus E } \frac{dy}{|x-y|^{n+\beta}}&=\int_{(B \setminus E)\cap B(x,\rho) } \frac{dy}{|x-y|^{n+\beta}}+\int_{(B \setminus E)\cap B^{\complement(x,\rho)}} \frac{dy}{|x-y|^{n+\beta}}\\
&\geq \int_{(B \setminus E)\cap B(x,\rho) } \frac{dy}{\rho^{n+\beta}}+\int_{(B \setminus E)\cap B^{\complement}(x,\rho)} \frac{dy}{|x-y|^{n+\beta}}\\
& =\frac{|(B \setminus E)\cap B(x,\rho)|}{\rho^{n+\beta}}+\int_{(B \setminus E)\cap B^{\complement}(x,\rho)} \frac{dy}{|x-y|^{n+\beta}}\\
& \geq \frac{|E \cap  B^{\complement}(x,\rho)|}{\rho^{n+\beta}}+\int_{(B \setminus E)\cap B^{\complement}(x,\rho)} \frac{dy}{|x-y|^{n+\beta}}\\
& \geq \int_{|E \cap  B^{\complement}(x,\rho)|} \frac{dy}{|x-y|^{n+\beta}} +\int_{(B \setminus E)\cap B^{\complement}(x,\rho)} \frac{dy}{|x-y|^{n+\beta}}\\
& \geq \int_{B^{\complement}(x,\rho) \cap B} \frac{dy}{|x-y|^{n+\beta}}.
\end{align*}

If $B \cap (B(x,3\rho)\setminus B(x,2\rho) )= \emptyset$, then $B \cap B(x,2\rho)= B$. Thus
$$|B^{\complement}(x,\rho) \cap B| = |B \cap (B(x,3\rho)\setminus B(x,2\rho)|= |B|- |B \cap B(x,\rho)| > |E|.$$  Moreover,  there exists $\theta> 0$ independent of $B$ such that $|E|=|B \cap B(x,\rho)| \geq \theta \omega_n \rho ^{n}$. So we have

\begin{align*}
\int_{B^{\complement}(x,\rho) \cap B} \frac{dy}{|x-y|^{n+\beta}}
= \int_{ B \cap [B(x,2\rho)\setminus B(x,\rho)]} \frac{dy}{|x-y|^{n+\beta}}
\geq \frac{|E|^{-\beta/n}}{(\omega_n \theta)^{1+\frac{\beta}{n}}2^{n+\beta}}.
\end{align*}

If $B \cap (B(x,3\rho)\setminus B(x,2\rho) )\neq \emptyset$, then there exists a point $z \in B \cap (B(x,3\rho)\setminus B(x,2\rho) )$ such that \begin{align*}
B \cap B(z,\rho ) \subset B \cap [B(x,4\rho)\setminus B(x,\rho)] \subset B \setminus B(x,\rho)
\end{align*}
and
$$|B \cap [B(x,4\rho)\setminus B(x,\rho)] | \geq |B(z,\rho)\cap B| \geq \frac{\theta}{2}\omega_n \rho^{n}.$$
Therefore,
\begin{align*}
\int_{B^{\complement}(x,\rho) \cap B} \frac{dy}{|x-y|^{n+\beta}}
\geq \int_{ B \cap [B(x,4\rho)\setminus B(x,\rho)]} \frac{dy}{|x-y|^{n+\beta}}
\geq \frac{|E|^{-\beta/n}}{(\omega_n \theta)^{1+\frac{\beta}{n}}4^{n+\beta}}.
\end{align*}
This completes the proof of Lemma \ref{l3.5}.
\end{proof}

\begin{lem}\label{l3.6} Let $\beta \in (0,n)$ and  $\phi$ be a Young  function satisfying \eqref{delta0} and the doubling property with a constant $K$. There exists a positive constant $C(n,\beta,K)$ depending only on $n,\beta $ and $K$ such that for any ball $B$ and $u \in \dot W^{\bz,\phi}(B)$, we have
\begin{align*}
\|u-u_{B}\|_{L^{\phi^{n/(n-\bz)}}(B)}\le C(n,\beta,K) \|u\|_{\dot W^{\bz,\phi}(B)}.
\end{align*}

\end{lem}

\begin{proof}
Without loss of generality, we may assume $u \in L^{\infty}(B)$.
Indeed, let
\begin{equation}\label{s3.w3}
u_{N}:= \max \left\{ \min \{u(x), \, N\}, -N  \right\} \,\, \forall x \in B.
\end{equation}
By Lebesgue's convergence theorem, we   have $
\lim \limits _{N\rightarrow \infty} \|u_N\|_{L^{\phi^{n/(n-\beta)}}(B)} = \|u\|_{L^{\phi^{n/(n-\beta)}}(B)}
$
and $
\lim \limits _{N\rightarrow \infty} \|u_N\|_{W^{\bz,\phi} (B)} = \|u\|_{W^{\bz,\phi} (B)}
$.
%
If $\|u_{N}-(u_N)_{B}\|_{L^{\phi^{n/(n-\bz)}}(B)}\le C \|u_{N}\|_{\dot W^{\bz,\phi}(B)} $ holds for all $N$, sending $N\to\fz$ and noting $(u_N)_B\to u_B$, we have desired result.

Set the  median value $$m_{u}(B):= \inf \left\{c \in \rr :|\{x \in B: u(x)-c  > 0\}| \le \frac12 |B|\right\}.$$  Then
$$|\{x \in B: u(x)-m_{u}(B)  > 0\}| \le \frac12 |B|  \quad and \quad |\{x \in B: u(x)-m_{u}(B) <0\}| \le \frac12 |B|.$$
Write $u_+= [u-m_{u}(B)]\chi _{u \geq m_{u}(B)}$ and  $u_- = -[u-m_{u}(B)]\chi _{u \leq m_{u}(B)}$. We know $u-m_{u}(B)= u_+ - u_-$.
Note that
$$\|u-u_{B}\|_{L^{\phi^{n/(n-\bz)}}(B)} \leq 2 \|u-m_{u}(B)\|_{L^{\phi^{n/(n-\bz)}}(B)},
$$
and
$$\|u-m_{u}(B)\|_{L^{\phi^{n/(n-\bz)}}(B)} \leq C (\|u_{+}\|_{L^{\phi^{n/(n-\bz)}}(B)}+\|u_{-}\|_{L^{\phi^{n/(n-\bz)}}(B)}).
$$
Since
$$|u(x)-u(y)|= |[u(x)-m_{u}(B)]-[u(y)-m_{u}(B)]|=|u_+(x)-u_+(y)|+|u_-(x)-u_-(y)| \quad \forall x,y \in B,
$$
to get $\|u-u_{B}\|_{L^{\phi^{n/(n-\bz)}}(B)}\le C \|u\|_{\dot W^{\bz,\phi}(B)}$, it suffices to show
$
\|u_{\pm}\|_{L^{\phi^{n/(n-\bz)}}(B)} \leq C \|u_{\pm}\|_{\dot{W}^{\beta, \phi}(B)}.
$
Below we only prove this for $u_{+}$; the proof of $u_-$ is similar.
It suffices to find a constant $C(n,\beta, K)\ge 1$ such that
 for any $\lambda >4   C(n,\beta, K)  \|u\|_{ W ^{\beta,\phi}(\rn)}$,
 \begin{align}\label{3.yy3}
\int_{B} \phi^{n/(n-\beta)} \left(\frac{|u_{+}(x)|}{\lambda}\right)\,dx \leq 1.
\end{align}

To see \eqref{3.yy3}, for $k\in\zz$, define \begin{equation}\label{s3.w2}
\mbox{$A_{k } : = \{x \in B:u(x)_{+} > 2^{k}\}$ and $D_{k} := A_{k} \backslash A_{k+1}=\{x \in B: 2^{k} < u_{+}(x) \leq 2^{k+1}\} $.}
\end{equation}
Write $a_k:=|A_k|$ and $d_k:=|D_k|$.
For any $\lz>0$, we have
\begin{align*}
T:=\int_{B} \int_{B} \phi \left(\frac{|u_{+}(x)-u_{+}(y)|}{\lambda}\right) \, \frac{dxdy}{|x-y|^{n+\beta}} \geq
2\sum_{i\in\zz}\sum_{j\le i-2} \int_{D_i}\int_{D_j} \phi \left(\frac{|u_{+}(x)-u_{+}(y)|}{\lambda}\right) \, \frac{dxdy}{|x-y|^{n+\beta}}.
\end{align*}
Note that
$$
|u_{+}(x)-u_{+}(y)| \geq 2^i -2^{j+1} \geq 2^i -2^{i-1}=2^{i-1}
$$
whenever $x \in D_i$ and $j \in Z$ with $j \leq i-2$. So  by Lemma \ref{l3.5}, we have
\begin{align*}
T &\geq
2\sum_{i\in\zz}\sum_{j\le i-2} \int_{D_i}\int_{D_j} \phi \left(\frac{2^{i-1}}{\lambda}\right) \, \frac{dy}{|x-y|^{n+\beta}}\,dx\\
  &\geq 2\sum_{i\in\zz}  \int_{D_i}\int_{B \setminus (A_{i-1})} \phi \left(\frac{2^{i-1}}{\lambda}\right) \, \frac{dy}{|x-y|^{n+\beta}}\,dx\\
  &\geq  2C(n,\beta)\sum_{i\in\zz,a_{i-1}\ne0}  \int_{D_i}\phi\left(\frac{2^{i-1}}{\lambda}\right) a_{i-1} ^{-\beta/n}    \, dx\\
  &\geq 2 C(n,\beta)\sum_{i\in\zz,a_{i-1}\ne0}   \phi\left(\frac{2^{i-1}}{\lambda}\right)d_i  a_{i-1} ^{-\beta/n}:=2 C(n,\beta) S.
\end{align*}

Next we show that
\begin{equation}\label{eq5.x1}
2S\ge  \sum_{i\in\zz,a_{i-1}\ne0}
 \phi \left(\frac{2^{i-1}}{\lambda}\right)a_{i-1}^{-\beta/n}a_i.
 \end{equation}
Since  $d_{i}= a_{i}- \sum \limits _{ l \geq i+1}d_{l}$,  one has
$$
S = \sum_{i\in\zz,a_{i-1}\ne0}
 \phi \left(\frac{2^{i-1}}{\lambda}\right)a_{i-1}^{-\beta/n}a_i- \sum_{i\in\zz,a_{i-1}\ne0}\sum \limits _{  l \geq i+1}\phi \left(\frac{2^{i-1}}{\lambda}\right)a_{i-1}^{-\beta/n}d_{l}.
$$
For $ l\ge i+1$, we have $a_{l-1}\le a_{i-1}$, in particular,  $a_{l-1}\ne 0$ implies $a_{i-1}\ne0$.
Thus  by the convexity of  $\phi$,
\begin{align*}
\sum \limits _{i \in Z , a_{i-1}\neq 0} \sum \limits _{ l \geq i+1} \phi \left( \frac{2^{i-1}}{\lambda}\right)a_{i-1}^{-\beta/n}d_{l}&\le
\sum \limits _{l \in Z , a_{l-1}\neq 0} \sum \limits _{ i\le l-1} \phi \left( \frac{2^{i-1}}{\lambda}\right)a_{i-1}^{-\beta/n}d_{l}\\
&\le
\sum \limits _{l \in Z , a_{l-1}\neq 0} \sum \limits _{ i\le l-1} \phi \left( 2^{i-l} \frac{2^{l-1}}{\lambda}\right)a_{l-1}^{-\beta/n}d_{l}\\
&\le
\sum \limits _{l \in Z , a_{l-1}\neq 0} \sum \limits _{ i\le l-1} 2^{i-l} \phi \left( \frac{2^{l-1}}{\lambda}\right)a_{l-1}^{-\beta/n}d_{l}\\
& \le  \sum \limits _{l \in Z , a_{l-1}\neq 0}   \phi \left( \frac{2^{l-1}}{\lambda}\right)a_{l-1}^{-\beta/n}d_{l}=S,
\end{align*}
from which \eqref{eq5.x1} follows.

By \eqref{eq5.x1},  one has
\begin{equation*}
T\ge  C(n,\bz)\sum_{i\in\zz,a_{i-1}\ne0}
 \phi \left(\frac{2^{i-1}}{\lambda}\right)a_{i-1}^{-\beta/n}a_i= C(n,\bz)\sum_{i\in\zz,a_{k}\ne0}
 \phi \left(\frac{2^{k}}{\lambda}\right)a_{k}^{-\beta/n}a_{k+1}.
 \end{equation*}
Note that, using the H\"{o}lder inequality and $0<\bz<n$, noting $a_{k}=0$ implies $a_l=0$ for all $l\ge k$, we have
\begin{eqnarray*}
   \sum \limits _{k \in\zz} a_{k+1}^{1-\beta/n}\phi\left(\frac{2^{k}}{\lambda}\right)
&& = \sum \limits _{k \in\zz,a_k\neq 0} \left[ a_k^{(1-\beta/n)\beta/n}\phi\left(\frac{2^{k}}{\lambda}\right)^{\beta/n}\right]
\left[ a_{k+1}^{1-\beta/n}a_k^{-(1-\beta/n)\beta/n}\phi\left(\frac{2^{k}}{\lambda}\right)^{1-\beta/n}\right]\\
&& \leq \left[\sum \limits _{k \in\zz,a_k\neq 0 } a_k^{1-\beta/n}\phi\left(\frac{2^{k}}{\lambda}\right) \right]^{\beta/n}
\left[\sum \limits _{k \in\zz,a_k\neq 0} a_{k+1}a_k^{-\beta/n}\phi\left(\frac{2^{k}}{\lambda}\right)\right]^{1-\beta/n}\\
&&\le C(n,\bz)  T ^{1-\bz/n} \left[\sum \limits _{k \in\zz,a_k\neq 0 } a_k^{1-\beta/n}\phi\left(\frac{2^{k}}{\lambda}\right) \right]^{\beta/n}.
\end{eqnarray*}
By the doubling property of $\phi$, one has
\begin{eqnarray*}\sum \limits _{k \in\zz,a_k \neq 0} a_{k }^{1-\beta/n}\phi\left(\frac{2^{k}}{\lambda}\right)&&\le K
\sum \limits _{k \in\zz,a_k+1 \neq 0} a_{k+1 }^{1-\beta/n}\phi\left(\frac{2^{k}}{\lambda}\right). \end{eqnarray*}
 From this one concludes that
$$ T\ge C(n,\bz,K) \sum \limits _{k \in\zz} a_{k+1}^{1-\beta/n}\phi\left(\frac{2^{k}}{\lambda}\right).$$

On the other hand,
\begin{eqnarray*}
 \int_{B} \phi^{n/(n-\beta)} \left(\frac{|u_{+}(x)|}{4\lambda}\right)\,dx
 &&\le \sum \limits _{k \in Z} \int_{D_{k }}  \phi^{n/(n-\beta)} \left(\frac{2^{k-1 }}{ \lambda}\right) \,  dx
\leq   \sum \limits _{k \in Z}
\phi^{n/(n-\beta)} \left(\frac{2^{k-1}}{\lambda}\right)a_{k }.
\end{eqnarray*}
Since $n/(n-\beta )\ge1$, we obtain
$$ \int_{B} \phi^{n/(n-\beta)} \left(\frac{|u_{+}(x)|}{4\lambda}\right)\,dx
\le  \left[\sum \limits _{k \in Z}
\phi  \left(\frac{2^{k-1}}{\lambda}\right)a_{k } ^ {1-\beta/n} \right]^{n/(n-\bz)}
 \leq  [C(n,\beta,K)  T]^{n/(n-\bz)} .$$
 Up to considering $C(n,\beta,K)+1$, assume that $C(n,\beta,K)\ge 1$.
If $\lambda >4   C(n,\beta, K)  \|u\|_{ W ^{\beta,\phi}(\rn)}$,
by Lemma \ref{l2.4}, we have
\begin{align*}
\int_{B} \phi^{n/(n-\beta)} \left(\frac{|u_{+}(x)|}{\lambda}\right)\,dx
\leq &  \left[C(n,\beta,K)\int_{B} \int_{B} \phi \left(\frac{|u_{+}(x)-u_{+}(y)|}{\lambda/4}\right) \, \frac{dxdy}{|x-y|^{n+\beta}}\right]^{\frac{n}{n-\beta}}\\
\leq & \left[\int_{B} \int_{B} \phi \left(\frac{  |u_{+}(x)-u_{+}(y)|}{\lambda  /4C(n,\beta, K)}\right) \,\frac{dxdy}{|x-y|^{n+\beta}}\right]^{\frac{n}{n-\beta}} \leq 1,
\end{align*}
which gives \eqref{3.yy3}.
\end{proof}

\begin{proof}[Proof of Theorem \ref{l3.9}] Lemma \ref{l3.6} gives Lemma \ref{l3.9} for any ball $B$ of $\rn$.
We still need to consider the case $B=\rn$.
Given any $u \in \dot W^{\bz,\phi}(\rn)$, applying Lemma \ref{l3.6}, we have
\begin{align*}
\|u-u_{B}\|_{L^{\phi^{n/(n-\bz)}}(B)}\leq C \|u\|_{\dot W^{\bz,\phi}(B)} \lesssim  \|u\|_{\dot W^{\bz,\phi}(\rn)}, \quad \forall\, \, ball \, B \subset \rn.
\end{align*}

For any $k\in \zz$, using  Jensen's inequality, one obtains
\begin{align*}
\phi ^{n/(n-\bz)}\left(\frac{|u_{B(0,2^{k-1})}-u_{B(0,2^{k})}|}{\lambda}\right)
& \leq \fint_{B(0,2^{k-1})}\phi ^{n/(n-\bz)}\left(\frac{|u_(z)-u_{B(0,2^{k})}|}{\lambda}\right)  \,dz \\
& \leq \frac{1}{2^{kn}} \int_{B(0,2^{k})}\phi ^{n/(n-\bz)}\left(\frac{|u_(z)-u_{B(0,2^{k})}|}{\lambda}\right)  \,dz.
\end{align*}
By Lemma \ref{l2.4}, we get
\begin{align*}
|u_{B(0,2^{k-1})}-u_{B(0,2^{k })}|
& \leq \phi ^{-1} (2^{-k(n-\beta)}) \|u\|_{L^{\phi^{n/(n-\bz)}}(B(0,2^{k }))}\\
&\ls 2^{-\frac{k(n-\beta)}{K-1}}  \|u\|_{L^{\phi^{n/(n-\bz)}}(B(0,2^{k }))}
\ls 2^{-\frac{k(n-\beta)}{K-1}} \|u\|_{\dot W^{\bz,\phi}(\rn)}.
\end{align*}
This implies that $u_{B(0,2^{k})}$ converges to some  $c\in \rn$ as $k \rightarrow \infty$ and
\begin{align*}
|u_{B(0,2^{k})}-c|
&\leq \sum  \limits _{l \geq k} |u_{B(0,2^{l})}-u_{B(0,2^{l+1})}|\leq \sum  \limits _{l \geq k}2^{-\frac{l(n-\beta)}{K-1}}  \|u\|_{\dot W^{\bz,\phi}(\rn)}
\lesssim  2^{-\frac{k(n-\beta)}{K-1}}  \|u\|_{\dot W^{\bz,\phi}(\rn)}.
\end{align*}
Therefore,
\begin{align*}
\|u-c\|_{L^{\phi^{n/(n-\beta)}}(B(0,2^{k+1}))}
& \leq \|u-u_{B(0,2^{k})}\|_{L^{\phi^{n/(n-\beta)}}(B(0,2^{k+1}))} + \|u_{B(0,2^{k})}-c\|_{L^{\phi^{n/(n-\beta)}}(B(0,2^{k+1}))}
\lesssim  \|u\|_{ \dot W ^{\beta,\phi}(\rn)}.
\end{align*}
Letting $k \rightarrow \infty$, we get
$\inf \limits _{c\in\rr}\|u-c\|_{L^{\phi^{n/(n-\beta)}}(\rn )}\leq C \|u\|_{ \dot W ^{\beta,\phi}(\rn)}$.
This completes the proof of Theorem \ref{l3.9}.
\end{proof}

\begin{proof}[Proof of (ii) $\Rightarrow$ (iii) Theorem \ref{t1.2}: case $\bz\in(0,n)$]
Since $\Omega$ is a ${\dot W }^{\beta,\phi}$-extension domain,
for any  $u\in \dot{ W }^{\beta,\phi}(\Omega)$, we can find a $\tilde{u}\in \dot{ W }^{\beta,\phi}(\rn)$ such that $\tilde{u}=u$ in $\Omega$ and $\|\tilde{u}\|_{\dot{ W }^{\beta,\phi}(\rn)  }\le  C(\Omega)\|u\|_{\dot{ W }^{\beta,\phi}(\Omega)}$.
If $\beta < n $, applying Theorem \ref{l3.9}, we know
\begin{align*}
\inf \limits _{c\in\rr}\|u-c\|_{L^{\phi^{n/(n-\beta)}}(\Omega)}
\leq \inf \limits _{c\in\rr}\|\tilde{u}-c\|_{L^{\phi^{n/(n-\beta)}}(\rn )}
\leq C \|\tilde{u}\|_{ \dot W ^{\beta,\phi}(\rn)}
\lesssim \|u\|_{ \dot W ^{\beta,\phi}(\Omega)}
\end{align*}
as desired.
\end{proof}

\section{Proofs of Theorem \ref{t1.1} and    (i)$\Rightarrow$(ii) of Theorem \ref{t1.2}}\label{s4}

Since  (i)$\Rightarrow$(ii) of Theorem \ref{t1.2} follows from Theorem \ref{t1.1},
below we only need to prove Theorem \ref{t1.1}. To this end, we recall the properties of Ahlfors $n$-regular domains in Section \ref{s4.1}. The proof of  Theorem \ref{t1.2} is given in Section \ref{s4.2}.

\subsection{Some basic properties of Ahlfors $n-$ regular domains\label{s4.1}}
Let $\Omega$ be an Ahlfors $n$-regular domain, and write $U:=\rn\setminus\overline \Omega $. Observe that $|\partial\boz|=0$; see \cite{hkt08} and also \cite{z14}. Without loss of generality, we assume $U  \neq \emptyset$. Moreover, $\diam\Omega=\infty$ if and only if $|\Omega|
=\infty $; see \cite{lz18}.

It is well known that $U$ admits a Whitney decomposition $\mathscr W$.

\begin{lem}\label{wy}
There exists a collection $\mathscr W=\{Q_i\}_{i \in \nn}$ of   (closed) cubes  satisfying
\begin{enumerate}
\item[ (i)] $U = \cup_{i \in \nn}Q_i$, and  $ Q^\circ_{k}  \cap Q^\circ_{i}  =\emptyset $ for all $i, k \in \nn$ with $i\neq k$;

\item[ (ii)] $ \sqrt{n}l(Q_k) \leq \dist (Q_k , \partial \Omega) \leq 4\sqrt{n} l(Q_k)$;

\item[ (iii)] $\frac{1}{4} l(Q_k) \leq l(Q_i) \leq 4l(Q_k)$ whenever $Q_k \cap Q_i \neq \emptyset$.
\end{enumerate}
\end{lem}

For any $Q\in \mathscr W$, denote the neighbour cubes of $Q$ in $\mathscr W$ by $$N (Q)=\{P\in \mathscr W, P \cap Q \neq \emptyset\}.$$  There exists an integer $\gamma_0$ depending only on $n$ such that for all $Q \in \mathscr W  $,
\begin{eqnarray} \label{e2.w1}
\sharp N(Q) \leq \gamma_0 .
\end{eqnarray}
Moreover, for any $Q \in\mathscr W  $, we have
\begin{eqnarray} \label{e2.w2}
\frac1{|Q|}\int_U  \chi_{\frac{9}{8}Q}(x) \,dx\le 4^n\gamma_0.
\end{eqnarray}
Indeed,
$$ \frac1{|Q|}\int_U  \chi_{\frac{9}{8}Q}(x) \,dx\le \sum_{P\in \mathscr W}\frac1{|Q|}\int_P  \chi_{\frac{9}{8}Q}(x) \,dx.$$
Since  $P\cap  \frac{9}{8}Q\ne \emptyset$ implies that $P\subset N(Q)$ and $l_Q\le 4l_P$,
by $\sharp N(Q)\le \gamma_0$, we arrive at
$$\frac1{|Q|}\int_U  \chi_{\frac{9}{8}Q}(x) \,dx\le \sum_{P\in N(Q)} \frac{|P|}{|Q|} \le 4^n\gamma_0 $$
as desired.

For $\epsilon>0$, set
$$
\mathscr W_{\ez}:=\{Q  \in \mathscr W : l_Q <  \frac{1}{\epsilon}\diam \Omega \}.
$$
Obviously,  $\mathscr W= \mathscr W_\ez$ for any $\ez>0$ if $\diam\Omega=\fz$, and $\mathscr W_\ez\subsetneq \mathscr W$ for any $\ez>0$ if $\diam\Omega<\fz$.

For each $Q:= Q( x_Q ,l_Q  ) \in \mathscr W_\ez$ and let $ x_Q ^\ast \in \Omega$  be a point nearest to $  x_Q$ on $\overline \Omega$. By  Lemma \ref{wy} (ii),
we have
$$\widetilde Q^\ast:=Q(  x^\ast_Q , l_Q ) \subset 10 \sqrt{n}Q.$$
 Furthermore, write a reflecting ''cubes'' of $Q$ as
$$
 \widetilde Q^{\ast\ez} := (\ez Q^\ast  \cap \Omega)\setminus   \left( \bigcup  \{{\epsilon}P^\ast: P \in \mathcal{A}^\ez_Q\}   \right).
$$
 The following lemma says that if $\ez $ is small enough, then   the reflecting ''cubes'' of
$\mathscr W_\ez$  enjoy  the following fine properties.

Recall that the reflecting cubes was constructed in \cite{s07}.

\begin{lem}\label{l2.1} let $\ez_0 = (\theta/2\gz_0)^{1/n}/ (30\sqrt{n} ) $ and write $Q^\ast= \widetilde Q^{\ast\ez_0}$ for each $Q\in \mathscr W_{\ez_0}$.
\begin{enumerate}
\item[  (i) ] $ Q^\ast\subset (10Q) \cap \Omega$ for all $Q\in \mathscr W_{\ez_0}$;

\item[  (ii) ] $|Q| \leq \gamma_1 |Q^\ast|$ whenever $Q\in \mathscr W_{\ez_0}$;
\item[  (iii) ] $\sum \limits _{Q \in \mathscr W_{\ez_0}} \chi_{ Q^\ast} \leq \gamma_2$.
\end{enumerate}
\noindent
Above $\gamma_1$ and $\gamma_2$ are positive constants depending only on $n$ and $\theta$.
\end{lem}

If $\Omega$ is bounded, we let $Q^\ast=\Omega$ as the reflected  "cube"  of  any cube $Q\in\mathscr W\setminus \mathscr W_{\ez_0}\ne\emptyset$.
Write $$\mathscr W_{\ez_0}^{(k)}=\{Q\in N(P): P\in \mathscr W_{\ez_0}^{(k-1)}\}\quad\forall k\ge1,$$
where $\mathscr W_{\ez_0}^{^{(0)}}=  \mathscr W_{\ez_0} $. Namely, $\mathscr W_{\ez_0}^{(k)}$ is the $k^{\rm th}$-neighbors of $ \mathscr W_{\ez_0} $. Meanwhile, we also write
\begin{align} \label{vk.1}
V^{(k)}:= \bigcup \{x\in Q; Q\in    \mathscr W_{\ez_0}^{(k)}\}\quad\forall k\ge0.
\end{align}

Since  $Q^\ast=\Omega $ for $Q\notin \mathscr W_{\ez_0}$, applying  Lemma \ref{l2.1} (iii), we have
 $$\sum\limits_{Q \in\mathscr W^{^{(k)}}_{\ez_0}}\chi_{Q^\ast}\le  \sum\limits_{Q \in\mathscr W  _{\ez_0}}\chi_{Q^\ast}+ \sharp(\mathscr W^{^{(k)}}_{\ez_0} \setminus \mathscr W _{\ez_0})\chi_{\Omega}
\le [\gamma_2 +\sharp(\mathscr W_{\ez_0}^{(k)}\setminus  \mathscr W_{\ez_0}  )]\chi_{\Omega}\quad\forall k\ge1.$$
For $Q\in \mathscr W^{^{(k)}}_{\ez_0} \setminus \mathscr W _{\ez_0}$, observe that $l_Q\ge \frac1{\ez_0}\diam\Omega$ and
$l_Q\le 4^k l_{P}\le \frac{4^k}{\ez_0}  \diam\Omega$ for some $P\in \mathscr W_{\ez_0}$.
Thus, by Lemma \ref{wy} (ii), we have  $$Q\subset Q(\bar x, \diam\Omega+8\sqrt n\frac{4^k}{\ez_0}\diam\Omega)$$ for any fixed $\bar x\in \Omega$, and hence
$$\sharp(\mathscr W_{\ez_0}^{(k)} \setminus \mathscr W_{\ez_0})\le ( 1+ 8\sqrt n\frac{4^{k  } }{\ez_0})^n\ez_0^n \le (\ez_0+4^{k+2}\sqrt{n})^n.$$
This yields that
\begin{equation}\label{eq2.xx2}\sum\limits_{Q \in \mathscr W_{\ez_0}^{(k)}}\chi_{Q^\ast}   \le \gamma_2+(\ez_0+4^{k+2}\sqrt{n})^n\quad\forall k\ge1.
\end{equation}

Associated to $\mathscr W$, one has the following partition of unit of $U$.
\begin{lem}\label{l4.3}
There exists a family  $\{\vz_Q:Q\in\mathscr W\}$ of functions such that

\begin{enumerate}
\item[(i)] for each $Q  \in\mathscr W$,   $0\le \varphi_Q  \in C^{\infty}_0(\frac{17}{16}Q)$;

 \item[(ii)]for each $Q  \in\mathscr W$, $|\nabla \vz_Q|\le L/l_Q$;
  \item[(iii)]  $\sum_{Q  \in \mathscr W }  \vz =\chi_{U} $ .
\end{enumerate}
\end{lem}

\subsection{Proof  of  Theorem \ref{t1.1} \label{s4.2}}

Let $\Omega$ be an Ahlfors  $n$-regular domain.  To obtain Theorem $\ref{t1.1}$, it suffices to prove the existence of a bounded linear operator $E: \dot{ W }^{\beta,\phi}(\Omega)\to \dot{W} ^{\beta,\phi}(\rn)$ such that $Eu|_\boz=u$ for all $u\in \dot{ W }^{\beta,\phi}(\Omega)$.
The linear operator $E$ is given as follows: for any $u \in \dot{ W }^{\beta,\phi}(\Omega)$  define
\begin{equation*}
Eu(x)\equiv\lf\{\begin{array}{ll}
u(x),&x\in \boz,\\
0&x\in\partial \boz,\\
 \sum\limits_{Q  \in W  }  \varphi_Q(x) u_{Q^\ast} , \quad &x\in  U.
\end{array}\r.
\end{equation*}
Obviously, $Eu|_{\Omega}= u $ on $\Omega$.

To prove the boundedness of $E$, we just show that
  there exists a constant $M >0$ such that for all $\lambda > M\|u\|_{ W ^{\phi}(\Omega)}$ ,
\begin{eqnarray*}
H(\lambda):=\int_{\rn} \int_{\rn} \phi \left(\frac{|Eu(x)-Eu(y)|}{\lambda}\right) \, \frac{dydx}{|x-y|^{n+\beta}}  \leq 1.
\end{eqnarray*}
If $\|u\|_{\dot{ W }^{\beta,\phi}(\Omega)}=0$, then   $u$ and hence $Eu$ must be a  constant function essentially.
So we may assume that $\|u\|_{\dot{ W }^{\beta,\phi}(\Omega)}>0$; and moreover, we further assume that $\|u\|_{\dot{ W }^{\beta,\phi}(\Omega)}=1$ on account of the linearity of $E$.

For $\lambda >0$, write
\begin{eqnarray*}
H(\lambda)
&&= \int_{\Omega} \int_{\Omega} \phi \left(\frac{|u(x)-u(y)|}{\lambda }\right) \, \frac{dydx}{|x-y|^{n+\beta}}+
2\int_{U} \int_{\Omega} \phi \left(\frac{|Eu(x)-u(y)|}{\lambda }\right) \, \frac{dydx}{|x-y|^{n+\beta}}\\
&& \quad+\int_{U} \int_{U} \phi \left(\frac{|Eu(x)-Eu(y)|}{\lambda}\right) \, \frac{dydx}{|x-y|^{n+\beta}}\\
&& =: H_1(\lambda) + 2H_2(\lambda) +H_3(\lambda).
\end{eqnarray*}
 It  suffices to find constants $  M_i \ge 1$ depending only on $n$, $\theta$ and $\phi$ such that $H_i(\lambda) \le\frac{1}{4}$   whenever $ \lambda \ge  M_i $ for $ i=1,2,3$.
In fact, by taking $M=M_1+M_2+M_3$, we have $H(\lambda)\le1$ whenever $\lambda \ge M$.

Firstly, let $ M_ 1=4$. Then for $\lambda> 4$,
by the convexity of $\phi$ and $\|u\|_{\dot{ W }^{\beta,\phi}(\Omega)}=1$, we have
$$H_1(\lambda)\le \frac14\int_{\Omega} \int_{\Omega} \phi \left(\frac{|u(x)-u(y)|}{\lambda/4 }\right) \, \frac{dydx}{|x-y|^{2n}}\le \frac14.$$

In order to find $ M_ 2$ and $ M_ 3$, we think about two cases: $\diam\Omega=\fz$ and $\diam\Omega<\fz$.

\medskip
\noindent
{\it Case $\diam\Omega=\fz$.}
To find $ M_ 2$,
for any $x\in U $  and  $y\in\boz$, by Lemma \ref{l4.3}(iii),  one has
$$
Eu(x)-u(y)=\sum\limits_{Q  \in  \mathscr W   }\varphi_Q (x)[u_{Q^\ast}-u(y)],$$
 and hence,
using the convexity of $\phi$ and  Jensen's inequality,
\begin{eqnarray*}
\phi\left(\frac{|Eu(x)-u(y)|}{\lambda }\right)&&
\le \phi\left(\sum\limits_{Q  \in \mathscr  W   }\varphi_Q(x) \frac{|u_{Q^\ast}-u(y)|} {\lambda }\right)\\
 && \le   \sum\limits_{Q  \in \mathscr W  } \varphi_Q(x) \phi  \left(\fint_{Q^\ast} \frac{|u(z)-u(y)|}{\lambda}\,dz\right) \le  \sum\limits_{Q  \in \mathscr  W  }  \varphi_Q(x)  \fint _{Q^\ast}\phi\left(\frac{|u(z)-u(y)|}{\lambda }\right) \,dz.
\end{eqnarray*}
If $
\varphi_Q(x)\ne 0$, then $x\in \frac{17}{16}Q$. For $z\in Q^\ast$, by $Q^\ast\subset 10\sqrt n Q$,
we have
 $  |x-z| \leq  20  n l(Q)$.
If    $|x-y| \geq d(x,\, \Omega) \geq  l(Q)$,
 we know
 $|x-z|\le 20  n |x-y|$, that is , $$|y-z| \leq |x-y| + |x-z| \leq 21  n |x-y|.$$
Thus
\begin{eqnarray*}
\int_\boz \phi\left(\frac{|Eu(x)-u(y)|}{\lambda}\right)  \, \frac{dy}{|x-y|^{\beta+n}}
&& \le  (21  n)^{\beta+n}\sum\limits_{Q \in \mathscr W}\varphi_Q(x) \fint _{Q^\ast} \int_\boz\phi\left(\frac{|u(z)-u(y)|}{\lambda}\right)  \, \frac{dzdy}{|z-y|^{\beta+n}}.
\end{eqnarray*}
By Lemma \ref{l2.1} (ii), we get
 \begin{eqnarray*}
H_2(\lambda)
&& \leq  2(21  n)^{\beta+n} \int_U  \sum\limits_{Q \in \mathscr W}\varphi_Q(x)  \fint _{Q^\ast} \int_{\Omega}\phi\left(\frac{|u(z)-u(y)|}{\lambda}\right)  \frac{dzdy}{|y-z|^{\beta+n}} \,dx\\
&& \leq 2\gamma_1(21  n)^{\beta+n}   \sum\limits_{Q \in \mathscr W}\left(\frac1{|Q|}\int_U\varphi_Q(x) \,dx\right) \int _{Q^\ast} \int_{\Omega}\phi\left(\frac{|u(z)-u(y)|}{\lambda}\right)  \frac{dzdy}{|y-z|^{\beta+n}} .
\end{eqnarray*}
For  $\varphi_Q  \le \chi_{\frac98Q} $  as given in Lemma  \ref{l4.3},    by \eqref{e2.w2}
we have
$$\frac1{|Q|}\int_U  \vz_Q(x) \,dx\le \frac1{|Q|}\int_U  \chi_{\frac98Q}(x) \,dx  \le 4^n\gamma_0,$$
which implies that
 \begin{eqnarray}\label{eq4.w2}
H_2(\lambda)
&& \leq 2\gamma_1 4^n\gamma_0(21 n)^{\beta+n}  \sum\limits_{Q \in \mathscr W} \int _{Q^\ast}  \int_{\Omega}\phi\left(\frac{|u(z)-u(y)|}{\lambda}\right)   \frac{dzdy}{|y-z|^{\beta+n}} .
\end{eqnarray}
By $\sum_{Q\in \mathscr W}\chi_{Q^\ast}\le \gamma_2$(see Lemma \ref{l2.1} (iii)), we obtain
 \begin{eqnarray*}
H_2(\lambda)
&& \leq 2\gamma_1 4^n\gamma_0  \gamma_2 (21  n)^{\beta+n}  \int _\Omega  \int_{\Omega}\phi\left(\frac{|u(z)-u(y)|}{\lambda}\right)  \frac{dzdy}{|y-z|^{\beta+n}}.
\end{eqnarray*}
Let $ M_ 2= 8\gamma_1 4^n\gamma_0  \gamma_2 (21  n)^{\beta+n}$.
By the convexity of $\phi$ again, $\lambda >   M_ 2  $ gives $H_2(\lambda)\le \frac{1}{4}$.

To find $ M_ 3$, for each $x \in U$ ,  set
\begin{eqnarray*}
X_1(x) := \left\{ y \in U : |x-y| \geq \frac{1}{132  n} \max\{ d(x, \Omega) , d(y, \Omega)\}\right\}\quad{\rm and} \quad
X_2(x) :=U\setminus X_1(x).
\end{eqnarray*}
Write
\begin{eqnarray*}
H_3(\lambda) && = \int_{U} \int_{X_1(x)} \phi \left(\frac{|Eu(x)-Eu(y)|}{\lambda }\right) \, \frac{dydx}{|x-y|^{\beta+n}} +\int_{U} \int_{X_2(x)} \phi \left(\frac{|Eu(x)-Eu(y)|}{\lambda }\right) \, \frac{dydx}{|x-y|^{\beta+n}}\\
&& =: H_{31}(\lambda) + H_{32}(\lambda)
\end{eqnarray*}

Below, we will find $M_{3i}$ so that if $\lambda> M_ {3i}$, then $H_{3i}\le \frac{1}{8}$ for $i=1,2$.
Note that letting $ M_ 3=\max\{ M_ {31}, M_ {32}\}$,
for $\lambda> M_ 3 $, we have $H_{3}(\lambda)\le \frac14$ as desired.

To find $ M_ {31}$, for $x\in U$ and $y\in X_1(x)$,  thanks to $$\sum \limits _{Q  \in \mathscr  W  } \varphi_Q (x)=\sum \limits _{P  \in \mathscr W  } \varphi_P (y)=1,$$ we obtain
 \begin{eqnarray*}
Eu(x)-Eu(y)&&=\sum\limits_{P \in \mathscr W  }\sum\limits_{Q \in \mathscr W  }\varphi_Q(x)\varphi_P(y)[u_{Q^\ast}-u_{P^\ast}]\\
&&= \sum\limits_{P \in \mathscr W  }\sum\limits_{Q \in \mathscr W  }\varphi_Q(x)\varphi_P(y)\fint_{Q^\ast} \fint_{P^\ast }[u(z)-u(w)]\,dzdw.
\end{eqnarray*}
Again, applying the convexity of $\phi$ and Jensen's inequality, one gets
\begin{eqnarray*}
\phi\left(\frac{|Eu(x)-Eu(y|}{\lambda}\right)
 && \le   \sum\limits_{Q  \in\mathscr W  } \sum\limits_{P \in\mathscr W  }\varphi_Q(x)\varphi_{P}(y) \phi  \left(\fint_{Q^\ast} \fint_{P^\ast }\frac{|u(z)-u(w)|}{\lambda }\,dzdw\right)\\
&& \leq \sum \limits _{Q \in\mathscr W }\sum \limits _{P \in\mathscr W }  \varphi_Q(x)\varphi_P(y)
\fint_{Q^\ast} \fint_{P^\ast} \phi \left(\frac{|u(z)- u(w)|}{\lambda }\right)\, dwdz.
\end{eqnarray*}

For $x\in Q$ and $z\in Q^\ast$,  $Q^\ast\subset 10 \sqrt n Q$ so that $|x-z|\le 10  n  l_Q\le 10  n d(x,\Omega)$. Similarly, for $y\in  P$,
 and $ w \in P^\ast$, we have $|y-w|\le 10  n  d(y,\Omega)$ as well. Since $y\in X_{1}(x)$ with
 $132  n|x-y| \geq  \max\{d(x, \Omega), d(y, \Omega)\}$, we further know
 $$ |z-w|\le |x-z|+ |x-y| + |y-w| \le   2641n |x-y|.$$
As a consequence,
\begin{eqnarray*}
H_{31}(\lambda)&& \le (2641 n )^{\beta+n}
\int_{U} \int_{X_1(x)} \sum \limits _{Q \in \mathscr W }\sum \limits _{P \in\mathscr W }  \varphi_Q(x)\varphi_P(y)
\fint_{Q^\ast} \fint_{P^\ast} \phi \left(\frac{|u(z)- u(w)|}{\lambda }\right)\, \frac{dwdz} {|z-w|^{\beta+n}} \, {dydx}.
\end{eqnarray*}
By  $|Q|\le \gamma_1|Q^\ast|$ and $|P|\le \gamma_1|P^\ast|$ ( see Lemma \ref{l2.1} (ii)), we   have
\begin{eqnarray*}
H_{31}(\lambda)
\le (2641n )^{\beta+n}\gz_1^{2}\sum\limits _{Q \in \mathscr W }\sum \limits _{P \in\mathscr W }
\left(\frac1{|Q|}\int_{U}\varphi_Q(x)dx \frac1{|P|}\int_{U} \varphi_P(y) dy\right)
\int_{Q^\ast} \int_{P^\ast} \phi \left(\frac{|u(z)- u(w)|}{\lambda }\right)\, \frac{dwdz} {|z-w|^{\beta+n}}.
\end{eqnarray*}
Appiying Lemma  \ref{l4.3} and \eqref{e2.w2} again, we have
$$\frac1{|Q|}\int_{U}\varphi_Q(x)\,dx \frac1{|P|}\int_{U} \varphi_P(y)\, dy\le( 4^n\gamma_0)^2.$$
 Thus
\begin{eqnarray*}
H_{31}(\lambda)
&&\le (2641 n )^{\beta+n}\gz_1^{2}  (4^n\gamma_0)^2 \sum\limits _{Q \in \mathscr W }\sum \limits _{P \in\mathscr W }
\int_{Q^\ast} \int_{P^\ast} \phi \left(\frac{|u(z)- u(w)|}{\lambda }\right)\, \frac{dwdz} {|z-w|^{\beta+n}}.
\end{eqnarray*}
Observing $\sum\limits _{Q \in \mathscr W }\chi_{Q^\ast}\le \gz_2$ given by Lemma \ref{l2.1} (iii), we  arrive at
\begin{eqnarray*}H_{31}(\lambda)
&&\le (2641n )^{\beta+n}\gz_1^{2} \gz_2^2 (4^n\gamma_0)^2
\int_{U} \int_{U} \phi \left(\frac{|u(z)- u(w)|}{\lambda }\right)\, \frac{dwdz} {|z-w|^{\beta+n}}.
\end{eqnarray*}
Taking $ M_ {31}=8   (2641 n )^{\beta+n}\gz_1^{2} \gz_2^2 (4^n\gamma_0)^2$,
if $\lambda >  M_ {31} $, by the convexity of $\phi$ once more,
we have  $H_{31}(\lambda)\le \frac{1}{8}$.

To find $M_{32}$, write
\begin{eqnarray*}
H_{32}(\lambda) && =  \int_{U}  \sum_{P\in\mathscr W}\int_{P\cap X_2(x)} \phi \left(\frac{|Eu(x)-Eu(y)|}{\lambda }\right) \, \frac{dy}{|x-y|^{n+\beta}} dx.
\end{eqnarray*}
By $\sum \limits _{Q  \in \mathscr W }\left[ \varphi_Q (x) - \varphi_Q (y)\right]= 0$, for any $x\in U$ and $y\in X_2(x) \cap P$, we have
  $$ Eu(x)-Eu(y)=\sum \limits _{Q  \in\mathscr W }\left[ \varphi_Q (x) - \varphi_Q (y)\right][u_{Q^\ast}-u_{P^\ast }].$$
Furthermore,  by $|\nabla \vz_Q|\le L/l_Q$, we obtain
   $$ |Eu(x)-Eu(y)|\le L\sum \limits _{Q  \in \mathscr W } \frac{|x-y|}{l_Q}[\chi_{\frac{17}{16}Q}(x)+\chi_{\frac{17}{16}Q}(y)]|u_{Q^\ast}-u_{P^\ast }|.$$

Since $y\in X_2(x)$  with $|x-y|\le \frac1{132n} \max\{d(x,\Omega),d(y,\Omega)\} $,
taking $\bar{y} \in \bar{\Omega}$ with $|y-\bar{y}|=d(y,\Omega)$,   we obtain
\begin{eqnarray*}
d(x,\Omega)\leq |x-\bar{y}| \leq |x-y|+|y-\bar{y}|
 \leq \frac{1}{132 n }d(x,\Omega)+ \frac{1+132 n}{132 n}d(y,\Omega),\nonumber
\end{eqnarray*}
which leads to $d(x,\Omega)\leq \frac{132 n+1}{132 n-1}d(y,\Omega)$.
Similarly,  we have $d(y,\Omega)\leq \frac{132 n+1}{132 n-1}d(x,\Omega)$. 
Therefore, $$|x-y|\le \frac1{132 n}\frac{132 n+1}{132 n-1}d(x,\Omega).$$

For $y\in  \frac{17}{16} Q$, we know $Q\in N(P)$ and
$$   d(y,\Omega)\le d(y,Q)+\max_{a\in Q}d(a,\Omega) \leq \frac{1}{16}\sqrt n l_{  Q} + 4\sqrt{n}l_{  Q} \le  \frac{65}{16}\sqrt nl_{  Q}.$$
This implies
$$|x-y|\le \frac1{132 n}\frac{132 n+1}{132 n-1}\times \frac{65}{16}\sqrt nl_{  Q}\le \frac{1}{132\sqrt{n}} l_{  Q},$$
and hence  $x \in \frac{9}{8}Q$, that is,
$$\chi_{\frac{17}{16} Q}(y)\le \chi_{\frac{9}{8} Q}(x).$$
Similarly, if $x\in \frac{17}{16}Q$, we also have $y\in \frac{9}8 Q$ and  $Q\in N(P)$.
We may further write
   $$ |Eu(x)-Eu(y)|\le 2L\sum \limits _{Q  \in N(P) } \frac{|x-y|}{l_Q} \chi_{\frac{9}{8}Q}(x)  |u_{Q^\ast}-u_{P^\ast }|.$$
By $\sum \limits _{Q  \in W }  \chi_{\frac{17}{16}Q}(x) \le \gamma_0$  
and the convexity of $\phi$,  we have
\begin{eqnarray*}
\phi \left(\frac{|Eu(x)-Eu(y)|}{\lambda }\right)
&& \le \frac1{ \gamma_0}\sum \limits _{Q  \in N(P) }  \chi_{\frac{9}{8}Q}(x) \phi\left(\frac{|x-y|}{l_Q}\frac{|u_{Q^\ast}-u_{P^\ast }|}{\lambda/ 2L \gamma_0}\right)\\
\end{eqnarray*}
and hence
  \begin{eqnarray*}
H_{32}(\lambda) &&\le   \frac1{  \gamma_0}\int_{U} \sum_{P\in\mathscr W}\int_{P\cap X_2(x)}  \sum \limits _{Q  \in N(P) }  \chi_{\frac{9}{8}Q}(x) \phi\left(\frac{|x-y|}{l_Q}\frac{|u_{Q^\ast}-u_{P^\ast }|}{\lambda/ 2L \gamma_0}\right)\, \frac{dydx}{|x-y|^{n+\beta}}\\
&&\le   \frac1{  \gamma_0}\int_{U}\sum_{P\in\mathscr W} \sum \limits _{Q  \in N(P)}  \chi_{\frac{9}{8}Q}(x)   \int_{P\cap X_2(x)}\phi\left(\frac{|x-y|}{l_Q}\frac{|u_{Q^\ast}-u_{P^\ast }|}{\lambda/ 2L \gamma_0}\right)\, \frac{dy}{|x-y|^{n+\beta}}\,dx.
\end{eqnarray*}
Note that for $ x \in \frac98Q$ and $y\in P\cap X_2(x)$,  together with $d(x,\Omega)\le 4\sqrt n l_Q$, we have
$$|x-y|\le \frac1{132 n}\frac{132 n+1}{132 n-1}d(x,\Omega)\le l_Q .$$
By the condition \eqref{delta0}, we get
 \begin{eqnarray*}
\int_{P\cap X_2(x)}\phi\left(\frac{|x-y|}{l_Q}\frac{|u_{Q^\ast}-u_{P^\ast }|}{\lambda/ 2L \gamma_0}\right)\, \frac{dy}{|x-y|^{n+\beta}}
&&\le n\omega_n\int_0^{  l_Q} \phi\left(\frac{t}{l_Q}\frac{|u_{Q^\ast}-u_{P^\ast }|}{\lambda/ 2L \gamma_0}\right) \frac{dt}{t^{\beta+1}}\\
&&\le C_\beta(  l_Q)^{-\beta }  n\omega_n \phi\left( \frac{|u_{Q^\ast}-u_{P^\ast }|}{ \lambda/ 2 L  \gamma_0}\right).  \\
\end{eqnarray*}
Since $\sharp N(Q)\le \gamma_0$, the above inequality  leads to
  \begin{eqnarray*}
H_{32}(\lambda)
 &&\le  n C_\beta  \frac1{  \gamma_0}  \omega_n\int_{U} \sum_{P\in\mathscr W}\sum \limits _{Q  \in N(P)}
  ( l_Q)^{-\beta }  \chi_{\frac{9}{8}Q}(x)  \phi\left( \frac{|u_{Q^\ast}-u_{P^\ast }|}{4\lambda/ 2L\gamma_0}\right)
 \,dx\\
 &&\le   n C_\beta  \frac1{  \gamma_0} \omega_n  \sum_{P\in\mathscr W} \sum \limits _{Q  \in N(P) }\left(\frac1{|Q|}\int_U \chi_{\frac{9}{8}Q}(x) \,dx\right)( l_Q)^{n-\beta } \phi\left( \frac{|u_{Q^\ast}-u_{P^\ast }|}{ \lambda/ 2L  \gamma_0}\right)\\
 &&\le   nC_\beta  \omega_n 4^{n}  \sum_{P\in\mathscr W}\sum \limits _{Q  \in N(P) }  ( l_Q)^{n-\beta }  \phi\left( \frac{|u_{Q^\ast}-u_{P^\ast }|}{ \lambda/ 2L \gamma_0}\right).
\end{eqnarray*}

For any  $P\in \mathscr W$ and $Q  \in N(P)$, using  Jessen's inequality, one has
$$
\phi\left( \frac{|u_{Q^\ast}-u_{P^\ast }|}{ \lambda/ 2 L  \gamma_0}\right) \le \fint_{Q^\ast}\fint_{P^\ast}\phi\left (\frac{|u(z)-u(w)|}{ \lambda/2 L \gamma_0 }\right)\,dz\,dw. $$
Observe that $z\in P^\ast\subset 10\sqrt{n}P$ and $w\in Q^\ast\subset 10\sqrt{n}Q$ provides
$$|z-w|\le 10n(l_Q+l_P)\le 50 n \min\{l_Q, l_P\}.$$

If $n-\beta <0$, then $ ( l_Q)^{n-\beta }\leq (50 n)^{\beta-n}\frac{1}{|z-w|^{\beta-n}}$ .
Since $|Q|\le \gamma_1|Q^\ast|$ and   $|P|\le \gamma_1|P^\ast|$, this implies that
$$|z-w|^{2n}\le   (50 n)^{2n} (\gamma_1)^2 |Q^\ast|  |P^\ast| .$$
Hence
$$
\phi\left( \frac{|u_{Q^\ast}-u_{P^\ast }|}{\lambda/ 2 L  \gamma_0}\right) \le    (50 n)^{n+\beta} (\gamma_1)^2 \int_{Q^\ast}\int_{P^\ast}
 \phi\left( \frac{|u(z)-u(w)|}{ \lambda/2  L \gamma_0}\right)\,\frac{dz\,dw}{|z-w|^{n+\beta}}.$$

If $n-\beta > 0$,  $|Q|\le \gamma_1|Q^\ast|$ implies $( l_Q)^{n-\beta }\leq ( \gamma_1  |Q^{*}|)^{1-\beta/n }$.
Since $|Q|\le \gamma_1|Q^\ast|$ and   $|P|\le \gamma_1|P^\ast|$, this implies that
$$|z-w|^{\beta+ n}\leq   (50 n)^{\beta+ n} (\gamma_1)^{1+\beta/n} |Q^\ast|^{\beta/n}  |P^\ast| .$$
Therefore,
$$
\phi\left( \frac{|u_{Q^\ast}-u_{P^\ast }|}{ \lambda/ 2 L  \gamma_0}\right) \le    (50 n)^{n+\beta} (\gamma_1)^2 \int_{Q^\ast}\int_{P^\ast}
 \phi\left( \frac{|u(z)-u(w)|}{ \lambda/2  L \gamma_0}\right)\,\frac{dz\,dw}{|z-w|^{n+\beta}}.$$

We conclude that
\begin{eqnarray*}
H_{32}(\lambda)
  &&\le   nC_\beta  \omega_n 4^{2n}  (50 n)^{n+\beta} (\gamma_{1})^{2} \sum_{P\in W}\sum \limits _{Q  \in N(P) }    \int_{Q^\ast}\int_{P^\ast}
 \phi\left( \frac{|u(z)-u(w)|}{ \lambda/2  L \gamma_0}\right)\,\frac{dz\,dw}{|z-w|^{n+\beta}}.
\end{eqnarray*}
Applying $\sum_{Q\in W}\chi_{Q^\ast}\le \gamma_2$ again, we get
   \begin{eqnarray*}
H_{32}(\lambda)
 &&\le nC_\beta  \omega_n 4^{2n} (50 n)^{n+\beta} (\gamma_1)^2(\gamma_{2})^2  \int_{\Omega}\int_{\Omega}
 \phi\left( \frac{|u(z)-u(w)|}{ \lambda/ 2L \gamma_0}\right)\,\frac{dz\,dw}{|z-w|^{n+\beta}}.
\end{eqnarray*}
Taking  $M_{32}=8L \gamma_0nC_\beta  \omega_n 4^{2n} (50\sqrt n)^{n+\beta} (\gamma_1)^2(\gamma_{2})^2  $,
if $\lambda>  M_{32} $, we have $H_{32}(\lambda)\le \frac18$ as desired.

\medskip
\noindent
{\it Case $\diam\Omega<\fz$}.
To find $M_2$,
write $H_2(\lambda)$ as
\begin{eqnarray*}
H_2(\lambda)
&&= \int_{V^{(2)}} \int_{\Omega} \phi \left(\frac{|Eu(x)-u(y)|}{\lambda}\right)  \frac{dydx}{|x-y|^{n+\beta}}\\
&&+\int_{U \backslash V^{(2)}} \int_{\Omega} \phi \left(\frac{| u _\Omega-u(y)|}{\lambda}\right) \frac{dydx}{|x-y|^{n+\beta}}
:=H_{21}(\lambda) + H_{22}(\lambda).
\end{eqnarray*}
It suffices to find $M_{2i}$ such that $H_{2i}(\lambda)\le\frac{1}{8}$ for $i=1,2$.
Note that letting $ M_ 2=\max\{ M_ {21}, M_ {22}\}$,
for $\lambda> M_ 2 $, we have $H_{2}(\lambda)\le \frac14$ as desired.

For any $ x\in U\setminus V^{(2)}$,   $x$ belongs to some $Q\in \mathscr  W\setminus \mathscr W_{\ez_0}^{(2)}$. Hence  $N(Q)\cap \mathscr W_{\ez_0}=\emptyset$ and $P^\ast=\Omega$ for all $P\in N(Q)$.
By $\sum_{P\in N(Q)}\vz_P(x)=\sum_{P\in \mathscr W}\vz_P(x)=1$, we have
$$Eu(x)=\sum_{P\in \mathscr W}\vz_P(x)u_{P^\ast}=\sum_{P\in N(Q)}\vz_P(x)u_{P^\ast}=u_\Omega.$$

To find $M_{22}$, using  Jensen's inequality, we obtain
 \begin{eqnarray*}
 H_{22}(\lambda)
 &&\leq  \int_{U \backslash V^{(2)}} \int_{\Omega}  \left[\fint _{\Omega} \phi\left(\frac{|u(z)-u(y)|}{\lambda}\right) \,dz \right] \frac{dydx}{|x-y|^{n+\beta}}\\
 &&\leq  \int_{\Omega} \left[\frac{|\diam \Omega|^{n+\beta}}{|\Omega|}  \int_{U \backslash V^{(2)}}\frac{dx}{|x-y|^{n+\beta}} \right] \left[\int _{\Omega} \phi\left(\frac{|u(z)-u(y)|}{\lambda}\right) \,\frac{dz}{|z-y|^{n+\beta}}\right] dy.
 \end{eqnarray*}
Clearly, for $y \in \Omega$, letting $Q\in \mathscr  W\setminus \mathscr W_{\ez_0}$ and $x \in Q$, one has
$$|x-y|\ge d(x,\Omega) \ge l_Q \ge \frac1{\ez_0}\diam\Omega.$$
Furthermore, since $\Omega$ is an Ahlfors $n-$regular domain, we have $|\Omega| \geq \theta |\diam \Omega|^{n}$ . This yields
 \begin{eqnarray*}
\frac{|\diam \Omega|^{n+\beta}}{|\Omega|}\int_{U \backslash V}\frac{dx}{|x-y|^{n+\beta}}
 &&\leq\frac1{\theta} |\diam \Omega|^{\beta} \int_{|x-y|> \frac{1}{\epsilon_0}  \diam \Omega}\frac{dx}{|x-y|^{n+\beta}} \\
 &&\leq \frac{n}{\theta}|\diam \Omega|^{\beta} \omega_n  \int^{\infty}_{\frac{1}{\epsilon_0} \diam \Omega} \frac{1}{r^{\beta+1}} dr
 \leq \omega_n\frac{n \epsilon_0^{\beta}}{\theta\beta}.
 \end{eqnarray*}
We then conclude
\begin{eqnarray*}
H_{22}(\lambda)&& \leq \omega_n \frac{n \epsilon_0^{\beta}}{\theta\beta}\int_{\Omega} \int_{\Omega}  \phi\left(\frac{|u(z)-u(y)|}{\lambda}\right) \frac{dzdy}{|y-z|^{n+\beta}}.
\end{eqnarray*}
Letting  $M_{22}=8 \omega_n \frac{n \epsilon_0^{\beta}}{\theta\beta}$,
 by the convexity of $\phi$,  for $\lambda >  M_{22} $ we have $H_{22}(\lambda)\le \frac{1}{8}$.

To find $M_{21}$, note that for $x\in V^{(2)}$,
 $$\sum \limits _{Q \in\mathscr  W } \varphi_Q (x)= \sum \limits _{Q \in\mathscr  W_{\ez_0}^{(2)} } \varphi_Q (x) =1.$$
By the same  argument as $H_2(\lambda)$ in the case $\diam\Omega=\fz$, one has
 \begin{eqnarray*}
H_{21}(\lambda)
&& \leq 2\gamma_1 4^n\gamma_0(21 n)^{n+\beta}  \sum\limits_{Q \in\mathscr  W^{(2)}_{\ez_0}} \int _{Q^\ast}  \int_{\Omega}\phi\left(\frac{|u(z)-u(y)|}{\lambda}\right)   \frac{dzdy}{|y-z|^{n+\beta}} .
\end{eqnarray*}
Moreover,
 $$ \sum\limits_{Q \in\mathscr  W^{(2)}_{\ez_0}}\chi_{Q^\ast}   \le \gamma_2+(\ez_0+64\sqrt{n})^n,$$
and hence,
 \begin{eqnarray*}
H_{21}(\lambda)
&& \leq 2\gamma_1 4^n\gamma_0  [\gamma_2+(\ez_0+64\sqrt{n})^n] (21 n)^{n+\beta}  \int _\Omega  \int_{\Omega}\phi\left(\frac{|u(z)-u(y)|}{\lambda}\right)  \frac{dzdy}{|y-z|^{n+\beta}}.
\end{eqnarray*}
Letting $ M_{21} =16\gamma_1 4^n\gamma_0  [\gamma_2+(\ez_0+64\sqrt{n})^n] (21 n)^{n+\beta}$, by the convexity of $\phi$ again, if $\lambda>  M_{21} $, we have $H_{21}(\lambda)\le \frac{1}{8}$ as desired.

 To find $M_3$, note that
\begin{eqnarray*}U\times U
&\subset [V^{(3)}\times V^{(3)}]\cup  [V^{(2)}\times ( U \backslash  V^{(3)})] \cup  [ ( U \backslash  V^{(3)})\times V^{(2)} ] \cup [(U \backslash  V^{(2)} )\times (U \backslash  V^{(2)})] .
\end{eqnarray*}
Write
\begin{eqnarray*}
H_{3}(\lambda)&&=  \int_{V^{(3)}} \int_{ { V ^{(3)}}}   \phi \left(\frac{|Eu(x)-Eu(y)|}{\lambda}\right) \, \frac{dydx}{|x-y|^{n+\beta}} \\
&&+2  \int_{V^{(2)} } \int_{(U \backslash  V^{(3)}) }   \phi \left(\frac{|Eu(x)-Eu(y)|}{\lambda}\right) \, \frac{dydx}{|x-y|^{n+\beta}}
+  \int_{U\setminus V^{(2)}} \int_{U\setminus V^{(2)}}   \phi \left(\frac{|Eu(x)-Eu(y)|}{\lambda}\right) \, \frac{dydx}{|x-y|^{n+\beta}}\\
&&=: H_{31}(\lambda)+H_{32}(\lambda)+H_{33}(\lambda).
\end{eqnarray*}

Observe that
$Eu(x)=Eu(y)=u_\Omega$ for $x,y\in U\setminus V$, this implies $H_{33}(\lambda)=0$. It only suffices to find $M_{3i}$ such that $H_{3i}(\lambda)\leq \frac{1}{8}$ for all $\lambda >M_{3i} $ $(i=1,2)$.

For $H_{31}(\lambda)$, similarly to $H_{3}(\lambda)$ in the case $\diam\Omega=\fz$,  taking $ M_{31}$ as $M_3$ with $\gamma_2$ replaced by $\gamma_2+(\ez_0+4^5\sqrt{n})^n\omega_n$,
we can prove that  if $\lambda \ge M_{31}$, then $H_{31}(\lambda)\le \frac{1}{8}$. Here we omit the details.

For $H_{32}(\lambda)$,  note that for $y\in U \backslash   V^{(3)}$,  by $Eu(y)=u_\Omega$, we have
 $$H_{32}(\lambda)=  \int_{V^{(2)}} \int_{ U \backslash   V^{(3)}  }   \phi \left(\frac{|Eu(x)-u_\Omega|}{\lambda}\right) \, \frac{dydx}{|x-y|^{n+\beta}}.$$
 By  Jessen's inequality, one has
 $$H_{32}(\lambda)\le   \int_{V^{(2)}} \int_{ U \backslash   V^{(3)}  } \, \frac{dy }{|x-y|^{n+\beta}}\fint_\Omega \phi \left(\frac{|Eu(x)-u(z)}{\lambda}\right) \,   dx\,dz  .$$
For any $x\in V^{(2)}$ and $y\in U \backslash   V^{(3)}$, if $Q\in \mathscr W_{\ez_0}^{(3)}\setminus \mathscr W_{\ez_0}^{(2)}$ and $y\in Q$, then
 $|x-y|\ge l(Q)\ge \frac1 {\ez_0}\diam \Omega$.
 Thus
 $$\int_{ U \backslash   V^{(3)}  } \, \frac{dy }{|x-y|^{n+\beta}}\le\frac{ n \ez_0^\beta \omega_n}{\beta} (\diam \Omega)^{-\beta}.$$
By $|\Omega|\ge \theta\diam\Omega$,  one   has
  $$H_{32}(\lambda)\le  \frac{n \ez_0^{ \beta}\omega_n}{\theta \beta} (\diam \Omega)^{-(\beta+n)}    \int_{V^{(2)}}  \int_\Omega \phi \left(\frac{|Eu(x)-u(z)|}{\lambda}\right) \,   dx\,dz .$$
For any $x\in V^{(2)}$ , there exists a $P_i\in \mathscr W^{(i)}_{\ez_0}$ such that
  $x\in P_2$ and $P_i\in N(P_{i-1})$ for $i=1,2$.
Together with  $l(P_0)\le \frac1{\ez_0}\diam\Omega$ and Lemma \ref{wy}, we know
  $l(P_2)\le 4^2 \frac1{\ez_0}\diam\Omega$.
Hence for $y\in \Omega$,
  $$|x-y|\le \dist(x,\Omega)+\diam\Omega\le  \diam P_2+\dist(P_2,\Omega)+ \diam\Omega
  \le  4^4 \frac1{\ez_0}\sqrt n \diam\Omega.$$
This yields
$$H_{32}(\lambda)\le  \frac{n \ez_0^{ \beta} \omega_n}{\theta\beta}  (4^4 \frac1{\ez_0}\sqrt n)^{n+\beta}    \int_{V^{(2)}}  \int_\Omega \phi \left(\frac{|Eu(x)-u(z)|}{\lambda}\right) \,   \frac{dx\,dz}{|x-z|^{n+\beta}}\le
     \frac{ 4^ {4(n+\beta)} n ^{  \frac{n+\beta}{2}+1} \omega_n}{\theta\beta \ez_0^{\beta}}   H_{21}(\lambda).$$
Letting $M_{32} = \frac{ 4^ {4(n+\beta)+2} n ^{  \frac{n+\beta}{2}+1} \omega_n}{\theta\beta \ez_0^{\beta}}   M_ {21}$, if $\lambda > M_{32}$, we have $H_{32}(\lambda)\le \frac{1}{8}$ as desired.
Then completes the proof of Theorem \ref{t1.1}.

\section{Proof of (iii)$\Rightarrow$(i) of Theorem $\ref{t1.2}$}\label{s5}

To prove  (iii)$\Rightarrow$(i) of Theorem $\ref{t1.2}$, we need the following estimates for test functions.
Below we write $B_{\Omega}(x,r)= B (x,r) \cap \Omega$.
For $x\in \Omega$ and $0<r<t<\diam\Omega$, set
$$
u_{x,r,t}(z)=\left\{\begin{array}{ll}1& z\in B_\Omega(x,r)\\
\frac{t-|x-z|}{t-r}\quad & z\in B_\Omega(x,t)\setminus B_\Omega(x,r)\\
0& z\in\Omega\setminus B_\Omega(x,t)
\end{array}\r.
$$
Then proof of Lemma \ref{l3.3} is similar to  Lemma 5.1 of \cite{lz18}. For reader's convenience, we give the details.

\begin{lem}\label{l3.3}
Let $\beta >0$ and $\phi$ be a Young function satisfying \eqref{delta0}. For all $\beta>0$, $x\in \Omega$ and $0<r<t<\diam\Omega$,  we have $u_{x,r,t}\in \dot{ W }^{\beta,\phi}(\Omega)$ with
$$\|u_{x,r,t}\|_{\dot{ W }^{\beta,\phi}(\Omega)}\le C\left [\phi^{-1}\left(\frac{(t-r)^\beta}{|B_\Omega(x,t)| }\right)\right]^{-1}.$$
\end{lem}

\begin{proof}
Write
\begin{align*}
\int_\Omega\int_\Omega\phi\left(\frac{|u_{x,r,t}(z)-u_{x,r,t}(w)|}\lambda\right)\frac{dzdw}{|z-w|^{n+\beta}}
&=\int_{B_\Omega(x,t)}\int_{B_\Omega(x,t)}\phi\left(\frac{|u_{x,r,t}(z)-u_{x,r,t}(w)|}\lambda\right)
\frac{dzdw}{|z-w|^{n+\beta}}\\
&\quad+\int_{\Omega\setminus B_\Omega(x,t)}\int_{B_\Omega(x,t)}\phi\left(\frac{|  u_{x,r,t}(z)|}\lambda\right)\frac{dzdw}{|z-w|^{n+\beta}}\\
&=H_1(\lambda)+H_2(\lambda).
\end{align*}
Clearly,
\begin{align*}H_2(\lambda)&\le \int_{B_\Omega(x,t) \setminus B_\Omega(x,r)}\phi\left(\frac{t-|z-x|}{\lambda(t-r)}\right)\int_{\Omega\setminus B_\Omega(x,t)}\frac{dw}{|z-w|^{n+\beta}}dz + \int_{B_\Omega(x,r) }
\int_{\Omega\setminus B_\Omega(x,t)}\phi\left(\frac{1}{\lambda }\right)\frac{dw}{|z-w|^{n+\beta}}dz .
\end{align*}
Since $ \Omega\setminus B_\Omega(x,t)\subset \Omega\setminus B_\Omega(z, t-|z-x|)$, then
$$\int_{\Omega\setminus B_\Omega(x,t)}\frac{dw}{|z-w|^{n+\beta}}\le \int_{\rn\setminus B_\Omega(z, t-|z-x|)}\frac{dw}{|z-w|^{n+\beta}}\le \frac{n}{\beta} \omega_n  (t-|z-x|)^{-\beta}.$$
This induces
\begin{align*}H_2(\lambda)&\le \int_{B_\Omega(x,t) \setminus B_\Omega(x,r)}\phi\left(\frac{t-|z-x|}{\lambda(t-r)}\right) \frac{n}{\beta}\omega_n  (t-|z-x|)^{-\beta}dz + \int_{B_\Omega(x,r) }
 \phi\left(\frac{1}{\lambda }\right) \frac{n}{\beta} \omega_n  (t-|z-x|)^{-\beta}dz \\
&\le   \frac{n}{\beta} \omega_n(\beta+1)\frac{|B_\Omega(x,t)|}{(t-r)^\beta}\left[
\sup_{s\in(0,1]}\phi\left(\frac{s}{\lambda }\right)\frac1{s^{ \beta}}  +\phi\left(\frac{1 }{\lambda }\right)\right].
\end{align*}
Moreover,
$$
\sup_{s\in(2^{-j-1},2^{-j}]}\phi\left(\frac{s}{\lambda }\right)\frac1{s^{ \beta}}\le  \frac{2^\beta}{\beta} \int_{2^{-j}}^{2^{-j+1}} \phi\left(\frac{s}{\lambda }\right)\frac{ds}{s^{ \beta+1}} , $$
which leads to
$$\sup_{s\in(0,1]}\phi\left(\frac{s}{\lambda}\right)\frac1{s^{ \beta}}\le \frac{2^\beta}{\beta} \int_0^2  \phi\left(\frac{s}{\lambda }\right)\frac{ds}{s^{ \beta+1}}\le \frac{4^\beta }{\beta}\int_0^1  \phi\left(\frac{2 s}{\lambda}\right)\frac{ds}{s^{ \beta +1}}\le \frac{4^\beta}{\beta} \phi\left(\frac{2}{\lambda }\right).$$
Therefore,
\begin{align*}
H_2
&\le  (\frac{4^\beta}{\beta}+1)\frac{n}{\beta} \omega_n(\beta+1)\frac{|B_\Omega(x,t)|}{(t-r)^\beta} \phi\left(\frac{2 }{\lambda } \right).
\end{align*}
If $\lambda =M[\phi^{-1}\left(\frac{(t-r)^\beta}{|B_\Omega(x,t)| }\right)]^{-1}$ and $M\ge 4  (\frac{4^\beta}{\beta}+1) \frac{n}{\beta} \omega_n(\beta+1)$, we have $H_2(\lambda)\le  \frac{1}{2} $.

Moreover,
\begin{align*}H_1(\lambda)&\le \int_{B_\Omega(x,t)}\int_{B_\Omega(w,t-r)}\phi\left(\frac{|z-w|}{\lambda(t-r)}\right)\frac{dz}{|z-w|^{n+\beta}}dw
+\int_{B_\Omega(x,t)}\int_{B_\Omega(w,2t)\setminus B_\Omega(w,t-r)}\phi\left(\frac{ 1}{\lambda }\right)\frac{dz}{|z-w|^{n+\beta}}dw.
\end{align*}
Observe that
\begin{align*} \int_{B_\Omega(w,t-r)}\phi\left(\frac{|z-w|}{\lambda(t-r)}\right)\frac{dz}{|z-w|^{n+\beta}}
&\le n\omega_n \int_0^{t-r}  \phi\left(\frac{s}{\lambda(t-r)}\right)\frac{ds}{s^{\beta+1}} \\
&\le n\omega_n (t-r)^{-\beta}\int_0^{1}  \phi\left(\frac{s}{\lambda}\right)\frac{ds}{s^{\beta+1}}.
\end{align*}
Applying the condition \eqref{delta0}, we get
\begin{align*} \int_{B_\Omega(w,t-r)}\phi\left(\frac{|z-w|}{\lambda(t-r)}\right)\frac{dz}{|z-w|^{n+\beta}}
&\le n\omega_n  ( t-r)^{-\beta}  \phi\left(\frac{ 1  }{\lambda }\right) .
\end{align*}
On the other hand,
\begin{align*} \int_{B_\Omega(w,2t)\setminus B_\Omega(w,t-r)}  \frac{dz}{|z-w|^{n+\beta}}  &\le \int_{\rn \setminus B_ (w,t-r)}  \frac{dz}{|z-w|^{n+\beta}} =  \frac{n}{\beta}\omega_n (t-r)^{-\beta}.
\end{align*}
Hence
$$H_1(\lambda)\le \frac{n}{\beta}\omega_n C_{\beta}(\beta+1)\frac{|B_\Omega(x,t)|}{(t-r)^\beta}\phi\left(\frac{1 }{\lambda }\right).$$
If $\lambda=M[\phi^{-1}\left(\frac{(t-r)^\beta}{|B_\Omega(x,t)| }\right)]^{-1}$ and $M\ge 2\frac{n}{\beta}(\beta+1)\omega_n C_{\beta}$, we have $H_1(\lambda)\le \frac{1}{2} $.

\end{proof}

We are ready to prove (iii)$\Rightarrow$(i) of Theorem \ref{t1.2}.

\begin{proof} [Proof of (iii)$\Rightarrow$(i) of Theorem \ref{t1.2}]
Below we consider two cases:  $\beta> n$ and $0<\bz<n$.

{ \it Case   $\beta> n$. }
Let $\Omega $ be a ${\dot W }^{\beta,\phi}-$ imbedding domain.
For any continuous function $u\in {\dot W }^{\beta,\phi}(\Omega)$, we have
\begin{align}\label{e4.yy1}
|u(x)- u(y)| \leq  C\phi^{-1}\left(|x-y|^{\beta-n}\right)\|u\|_{\dot{ W }^{\beta,\phi}(\Omega)} \quad \mbox{for almost all $x,y \in \Omega$}.
\end{align}
Given any $x\in \Omega$ and $0<r<t<\diam\Omega$,
let $u=u_{x,r,t}$ be as in as Lemma \ref{l3.3}.
Applying Lemma \ref{l3.3}, we have
\begin{eqnarray*}
|u(x)- u(y)| \leq C \phi^{-1}\left( t^{\beta-n}\right) \left[\phi^{-1}\left(\frac{(t-r)^\beta}{|B_\Omega(x,t)| }\right)\right]^{-1}.
\end{eqnarray*}
Without loss of  generality, we may assume $C \geq 1$.

On the other hand, let $r=t/2$ .
For $y \in B_{\Omega}(x,\, t+t/2) \backslash B_{\Omega}(x,\, t)$, we have  $|u(x)-u(y)|=1$.
Thus
\begin{eqnarray*}
\phi^{-1}\left(\frac{(t/2)^\beta}{|B_\Omega(x,t)| }\right)\leq C \phi^{-1}\left( t^{\beta-n}\right).
\end{eqnarray*}
By the doubling property of $\phi$,
we have
\begin{eqnarray*}
\frac{(t/2)^\beta}{|B_\Omega(x,t)| } =\phi\left[\phi^{-1}\left(\frac{(t/2)^\beta}{|B_\Omega(x,t)| }\right)\right]
&&\leq \phi \left[C\phi^{-1}\left( t/2^{\beta-n}\right) \right]
\leq C^{K}\phi \left[\phi^{-1}\left( t^{\beta-n}\right) \right]
 \leq C^{K}t^{\beta-n},
\end{eqnarray*}
namely,
$t^{n} \leq  2^{\beta}C^{K}|B_\Omega(x,t)|$ as desired.

{\it Case   $0<\beta < n $.}
Given any $0<t<\diam\Omega$, let $b_0=1$ and $b_j \in (0,1)$ for $j\in N$ such that
\begin{equation}\label{e5.w1}
|B(x, b_jt) \cap \Omega|= 2^{-1}|B(x, b_{j-1}t) \cap \Omega|= 2^{-j}|B(x, t) \cap \Omega|.
\end{equation}

For each $j\ge 1$, let $u_{x,b_{j+1}t, b_jt} $  as Lemma \ref{l3.3}. Note that $u_{x,b_{j+1}t, b_jt}-c \geq \frac{1}{2}$ either in $B_\Omega(x,b_{j+1}t)$ or in $\Omega \setminus B_\Omega(x,b_{j}t)$. For $j \geq 1$, it implies $ B_\Omega(x,b_{j-1}t) \setminus B_\Omega(x,b_{j}t) \subset \Omega \setminus B_\Omega(x,b_{j}t)$ and
$$
|\Omega \setminus B_\Omega(x,b_{j}t)| \geq | B_\Omega(x,b_{j-1}t) \setminus B_\Omega(x,b_{j}t)|= |B_\Omega(x,b_{j}t)|=2|B_\Omega(x,b_{j+1}t)|.
$$
Hence
\begin{align*}
\int_{\Omega} \phi^{n/(n-\beta)}\left(\frac{|u_{x,b_{j+1}t, b_jt}(z)-c|}{\lambda}\right)\, dz
\geq \int_{B_{\Omega}(x,b_{j+1}t)} \phi^{n/(n-\beta)}\left(\frac{1}{2\lambda}\right)\, dz
\geq |B_{\Omega}(x,b_{j+1}t)| \phi^{n/(n-\beta)}\left(\frac{1}{2\lambda}\right),
\end{align*}
that is, for $j \geq 1$,
\begin{align*}
\inf_{c\in\rr}\|u_{x,b_{j+1}t, b_jt}-c\|_{L^{\phi^{n/(n-\beta)}}(\Omega)} \geq 2\left[\phi^{-1}\left(\frac{1}{|B_\Omega(x,b_{j+1}t)|^{1-\beta/n} }\right)\right]^{-1}.
\end{align*}
On the other hand, by  \eqref{e4.yy1} and Lemma \ref{l3.3} one has
$$
\inf_{c\in\rr}\|u_{x,b_{j+1}t, b_jt}-c\|_{L^{\phi^{n/(n-\beta)}}(\Omega)}\leq C \|u_{x,b_{j+1}t, b_jt}\|_{\dot{ W }^{\beta,\phi}(\Omega)} \le C\left[\phi^{-1}\left(\frac{(b_{j}t-b_{j+1}t)^\beta}{|B_\Omega(x,b_{j}t)| }\right)\right]^{-1}.
$$
Thus  we conclude that
\begin{eqnarray*}
\phi^{-1}\left(\frac{(b_{j}t-b_{j+1}t)^\beta}{|B_\Omega(x,b_{j}t)| }\right)
\leq C\phi^{-1}\left(\frac{1}{|B_\Omega(x,b_{j+1}t)|^{1-\beta/n} }\right).
\end{eqnarray*}
Without loss of  generality, we may assume $C\geq 1$.
Applying  Lemma \ref{l2.4}, we know
\begin{eqnarray*}
\frac{(b_{j}t-b_{j+1}t)^\beta}{|B_\Omega(x,b_{j}t)| }&&=\phi\left[\phi^{-1}\left(\frac{(b_{j}t-b_{j+1}t)^\beta}{|B_\Omega(x,b_{j}t)| }\right)\right]
\leq \phi \left[C\phi^{-1}\left(\frac{1}{|B_\Omega(x,b_{j+1}t)|^{1-\beta/n} }\right)\right] \\
&&\leq   C  ^{K-1}
\phi \left[\phi^{-1}\left(\frac{1}{|B_\Omega(x,b_{j+1}t)|^{1-\beta/n}}\right)\right]
=C ^{K-1}\frac{1 }{|B_\Omega(x,b_{j+1}t)|^{1-\beta/n} }.
\end{eqnarray*}
Therefore,
\begin{align*}
(b_{j}t-b_{j+1}t)^{\beta}
&\leq C^{K-1} 2^{1-\beta/n}|B_\Omega(x,b_{j}t)|^{\beta/n}
 \leq C^{K-1} 2^{1-\frac{\beta}{n}(j+1)}|B_\Omega(x,t)|^{\beta/n}.
\end{align*}
Since $b_j \rightarrow 0$ as $ j \rightarrow \infty$, we have
\begin{eqnarray*}
b_1t = \sum \limits _{j\geq 1} (b_{j}t- b_{j+1} t)
\leq \sum \limits _{j \geq 1} C^{K-1} 2^{1-\frac{\beta}{n}(j+1)}|B_\Omega(x,t)|^{\beta/n}
 \lesssim |B_\Omega(x,t)|^{1/n}.
\end{eqnarray*}

Applying an argument similar to in \cite{hkt08}, both $b_1 \geq 1/10$ and  $b_1 \leq 1/10 $ satisfy $\left|B_{\Omega}(x,t)\right| \ge C t^n$ as desired.
\end{proof}

\medskip
 \noindent {\bf Acknowledgment}.
 The author would like thank the anonymous referee for several helpful suggestions and comments.
 The author  would like to thank  Professor Yuan Zhou for several valuable discussions of this paper.
The author is partially supported by National Natural Science Foundation of China (No. 11871088).

\medskip




\begin{thebibliography}{30}



\bibitem{bk95}
S. M. Buckley and P. Koskela, Sobolev-Poincar\'e implies John, Res.  Lett. 2 (1995), 577-594.


\bibitem {bsk96}
S. M. Buckley and P. Koskela, Criteria for imbeddings of Sobolev-Poincar\'e type, Internat. Math. Res. Notices 18(1996), 881-902.

\bibitem{ds93}
R. A. DeVore and R. C. Sharpley, Besov spaces on domains in $R^d$, Trans. Amer. Math. Soc. 335 (1993), 843-864.


\bibitem {f13}
C.  Fefferman,  A. Israel and G. K. Luli, Sobolev extension by linear operators, J. Amer. Math. Soc. 27 (2014),   69-145.

\bibitem{gkz13}
A. Gogatishvili, P. Koskela and Y. Zhou,   Characterizations of Besov and Triebel-Lizorkin spaces on metric measure spaces, Forum Math. 25(2013), 787-819.


\bibitem{h96}
P. Haj\l asz,  Sobolev spaces on an arbitrary metric space, Potential Anal. 5 (1996), 403-415.

\bibitem{hkt08}
P. Haj\l asz, P. Koskela and H. Tuominen, Sobolev imbeddings, extensions and measure density condition, J. Funct. Anal.  254  (2008),  1217-1234.

\bibitem{hkt08b}
P. Haj\l asz, P. Koskela and H. Tuominen, Measure density and extendability of Sobolev functions,
Rev. Mat. Iberoam.  24 (2008), 645-669.

\bibitem{j80}
P. W. Jones, Extension theorems for BMO, Indiana Univ. Math. J. 29 (1980), 41-66.

\bibitem{j81}
P. W. Jones, Quasiconformal mappings and extendability of functions in Sobolev spaces,
 Acta Math.  147  (1981), 71-88.


\bibitem{jw78}
A. Jonsson and  H. Wallin, A Whitney extension theorem in $L^p$ and Besov spaces,
 Ann. Inst. Fourier (Grenoble) 28 (1978),  139-192.

\bibitem{jw84}
A. Jonsson and H. Wallin, Function spaces on subsets of $\rn$, Mathematical Reports, (1)1984.


\bibitem{k98}
P. Koskela, Extensions and imbeddings,  J. Funct. Anal.  159
(1998), 369-384.
\bibitem{kxzz}
P. Koskela, J. Xiao, Y. Zhang and Y. Zhou, A quasiconformal composition problem for the $Q$-spaces, J. Eur. Math. Soc. 19(2017), 1159-1187

\bibitem{kzz}
P. Koskela, Y. Zhang and Y. Zhou, Morrey-Sobolev Extension Domains, J. Geom. Anal. 27(2017), 1413-1434.


\bibitem{mcp}
Matias Carrasco Piaggio, Orlicz spaces and the large scale geometry of Heintze groups, Mathematische Annalen 368( 2017), 433-481.

\bibitem{npv12}
E. Di Nezza, G. Palatucci, and E. Valdinoci, Hatchhiker's guide to the fractionalSobolev spaces, Buletin Des Science Matheematiques. 1136(2012), 521-573 .

\bibitem{R91}
M. Rao and Z. Ren, Theory of Orlicz spaces, Monographs and Textbooks in Pure and Applied Mathematics, vol.146. Marcel Dekker Inc, New York(1991)

\bibitem{r99}
V. S. Rychkov, On restrictions and extensions of the Besov and Triebel-Lizorkin spaces
with respect to Lipschitz domains,  J. London Math. Soc. (2)  60  (1999),  237-257.

\bibitem{s06}
P. Shvartsman, Local approximations and intrinsic
characterizations of spaces of smooth functions on regular subsets of $\rn$, Math. Nachr. 279 (2006), 1212-1241.

\bibitem{s07}
P. Shvartsman,  On extensions of Sobolev functions defined on regular subsets of metric measure spaces, Journal of Approximation Theory.  2144 (2007), 139-161.


\bibitem{s10}
P. Shvartsman,  On Sobolev extension domains in $R^n$, J. Funct. Anal.  258 (2010), 2205-2245.

\bibitem{s70}
E. M. Stein,  Singular integrals and differentiability properties of functions,
Princeton Mathematical Series, No. 30, Princeton University Press, Princeton, N.J. 1970

\bibitem{t02}
H. Triebel, Function spaces in Lipschitz domains and on Lipschitz manifolds,
Characteristic functions as pointwise multipliers,
Rev. Mat. Complut. 15 (2002), 475-524.

\bibitem{t08}
H. Triebel,  Function spaces and wavelets on domains, EMS Tracts in Mathematics, 7. European Mathematical Society (EMS), Z\"orich, 2008, x+256 pp.

\bibitem{lz18}
T. Liang and Y. Zhou,  Orlicz-Besov extension and Ahlfors $n$-regular domains, arxiv: 1901.06186.


\bibitem {yyz}
D. Yang, W. Yuan and Y. Zhou, Sharp boundedness of quasiconformal composition operators on Triebel-Lizorkin type spaces, J. Geom. Anal.  27 (2017), 1548-1588.


\bibitem {z11}
Y. Zhou, Haj\l asz-Sobolev extension and imbedding, J. Math. Anal. Appl. 382 (2011), 577-593.

\bibitem{z511}
Y. Zhou, Criteria for Optimal Global Integrablity of Haj\l asz-Sobolev Functions, Illinois Journal of Mathematics. 55 (2011), 1083-1103.

\bibitem {z14}
Y. Zhou, Fractional Sobolev extension and imbedding, Trans. Amer. Math. Soc. 367 (2015), 959-979.



\end{thebibliography}
\end{document}